\documentclass[11pt,reqno]{amsart}
\usepackage{amssymb,amsmath,amsfonts,amsthm,enumerate}
\usepackage{mathrsfs}
\usepackage{graphicx}

\usepackage[left=.75 in, right=.75 in,top=.75 in, bottom=.75 in]{geometry}
\usepackage{nameref,hyperref,cleveref}
\numberwithin{equation}{section}
\newtheorem{lem}{Lemma}[section]

\newtheorem{prop}[lem]{Proposition}
\newtheorem{thm}[lem]{Theorem}
\newtheorem{remark}[lem]{Remark}
\newtheorem{definition}[lem]{Definition}

\title{Bore wave solutions to the viscous Saint-Venant system with capillarity}

\author{Keming Meng}
\address{
Department of Mathematical Sciences\\
Carnegie Mellon University\\
Pittsburgh, PA 15213, USA
}
\email[K. Meng]{kemingm@andrew.cmu.edu}

\author{Ian Tice}
\address{
Department of Mathematical Sciences\\
Carnegie Mellon University\\
Pittsburgh, PA 15213, USA
}
\email[I. Tice]{iantice@andrew.cmu.edu}
\thanks{I. Tice was supported by an NSF Grant (DMS \#2508400), which provided summer research funds for K. Meng.}



\begin{document}

\begin{abstract} 
We study a traveling wave formulation of the viscous Saint-Venant system in one spatial dimension with capillarity, which can be reduced to an equivalent nonlinear third order ODE.  We construct heteroclinic orbits for small positive capillarity, for arbitrary positive capillarity under an explicit small Froude number condition, and in a small-amplitude regime away from the critical Froude number.  These then give rise to smooth traveling bore wave solutions to the viscous Saint-Venant system.
\end{abstract}

\maketitle

%%%%%%%%%%%%%%%%%%%%%%%%%%%%%%%%%%%%%%%%%%%%%%%%%%%%%%%%%%%%
%%%%%%%%%%%%%%%%%%%%%%%%%%%%%%%%%%%%%%%%%%%%%%%%%%%%%%%%%%%%
%%%%%%%%%%%%%%%%%%%%%%%%%%%%%%%%%%%%%%%%%%%%%%%%%%%%%%%%%%%%
\section{Introduction}
%%%%%%%%%%%%%%%%%%%%%%%%%%%%%%%%%%%%%%%%%%%%%%%%%%%%%%%%%%%%
%%%%%%%%%%%%%%%%%%%%%%%%%%%%%%%%%%%%%%%%%%%%%%%%%%%%%%%%%%%%
%%%%%%%%%%%%%%%%%%%%%%%%%%%%%%%%%%%%%%%%%%%%%%%%%%%%%%%%%%%%

The viscous Saint-Venant system, also known as the viscous shallow water system, is a nonlinear system of PDEs that describes the behavior of an incompressible viscous fluid in a regime in which the fluid's depth is significantly smaller than its characteristic horizontal scale.  It is derived from the free boundary incompressible Navier-Stokes system by way of an expansion in a ``shallowness'' parameter and a vertical averaging argument.  We refer to the surveys of Bresch \cite{Bresch_2009} and Mascia \cite{Mascia2010} for a detailed overview of the derivation, noting that more than one limit is possible depending on the precise modeling assumptions made; see also Marche \cite{Marche_2007}.    The specific system we study in this paper is a nondimensional traveling wave formulation of the system with incline, slip, and capillarity derived by Stevenson and Tice  \cite{Stevenson_Tice_bathymetry,Stevenson_Tice_bores} using the techniques from \cite{Bresch_2009,Mascia2010}:
\begin{equation}\label{vsv_tw_form}
 (4-U) H = A \text{ and } H(U-4)U' + a U -4 \mu (HU')' + g H H' - \sigma H H''' = 4aH,
\end{equation}
where $U : \mathbb{R} \to \mathbb{R}$ is the unknown horizontal velocity and $H : \mathbb{R} \to (0,\infty)$ is the unknown fluid height.  In this setting the traveling wave speed has been normalized to $4$.  The nondimensional parameters are as follows: $A \in (0,1)$ measures the relative horizontal flux in the traveling frame; $a>0$ is the nondimensional Navier slip coefficient, which controls bottom friction; $g \ge 0$ is related to the Froude number of the flow by $\mathrm{Fr}=\sqrt{8/g}$ and should be thought of as a measure of the strength of normal gravity; $\mu>0$ is the reciprocal Reynolds number of the flow; and $\sigma\ge0$ is the nondimensional capillarity coefficient, measuring the strength of surface tension relative to inertia.

Our main interest in this paper is bore waves, which at the level of free boundary Navier-Stokes are two-dimensional traveling wave solutions with distinct upstream and downstream shear flow limits.  With the convention that the waves travel in the positive horizontal direction, we say a bore wave is ebbing if the smaller asymptotic depth occurs at $-\infty$ and the larger at $+\infty$, and surging when this is reversed.  In the context of the viscous Saint-Venant system \eqref{vsv_tw_form} bores correspond to heteroclinic orbits connecting equilibrium states with different positive heights.  Stevenson and Tice \cite{Stevenson_Tice_bores} constructed both ebbing and surging viscous Saint-Venant bores with $\sigma=0$ and parameters $(g,A)$ belonging to an infinite measure subset of $[0,\infty) \times (0,1)$ (see Theorem \ref{main_no_st_bores} for a more precise statement and Figure \ref{parameter_regimes} for a graphical depiction).  These Saint-Venant bore profiles were then used as leading order approximations in a construction of bore wave solutions to free boundary Navier-Stokes in  sufficiently shallow fluid layers.  The paper \cite{Stevenson_Tice_bores} also establishes the convergence of the resulting Navier-Stokes solutions to the viscous Saint-Venant solutions,   thereby rigorously justifying the shallow water limit for these solutions.   Although surface tension is allowed in the Navier-Stokes construction of \cite{Stevenson_Tice_bores}, its contribution is of higher order in the chosen scaling and does not appear in the limiting Saint-Venant system.

Our aim in this paper is to construct bore wave solutions to \eqref{vsv_tw_form} when capillarity is retained by choosing $\sigma >0$.  The first equation in \eqref{vsv_tw_form} allows for the elimination of $U$ via  $U=4 - A/H$, which results in an equivalent third order ODE for $H>0$:
\begin{equation}\label{ss_height_equation}
    \sigma H H''' = 4a\left(1-H-\frac{A}{4H}\right)- 4\mu A\left(\frac{H'}{H}\right)'  - \left(\frac{A^2}{H^2}-gH\right)H',
\end{equation}
which still makes sense with $\sigma=0$, in which case it is second order.  Enforcing the condition $H>0$ suggests writing $H = e^\rho$ and reformulating in terms of the unknown $\rho : \mathbb{R} \to \mathbb{R}$:
\begin{equation}\label{ss_profile_equation}
    \sigma\rho''' = 4\mu A e^{-2\rho}\bigl(F(\rho)-G(\rho)\rho'-\rho''\bigr)
      -\sigma\bigl((\rho')^3+3\rho'\rho''\bigr),
\end{equation}
where $F,G : \mathbb{R} \to \mathbb{R}$ are the smooth functions defined by 
\begin{equation}\label{ss_coefficients}
    F(r)=\frac{a}{\mu A}\left(1-e^r-\frac{A}{4}e^{-r}\right)
    \text{ and }
    G(r)=\frac{1}{4\mu}\left(Ae^{-r}-\frac{g}{A}e^{2r}\right).
\end{equation}
The problem \eqref{ss_profile_equation} admits exactly two equilibrium solutions, $\rho= \rho_\pm$, corresponding to the two zeros of $F$:
\begin{equation}\label{ss_equilibrium_values}
    \rho_-=\log\left(\frac{1-\sqrt{1-A}}{2}\right),
    \text{ and }
    \rho_+=\log\left(\frac{1+\sqrt{1-A}}{2}\right).
\end{equation}
Writing $H_\pm=H_\pm(A)=e^{\rho_\pm}=(1\pm\sqrt{1-A})/2$, we see that the parameter $A$ determines the separation of the equilibrium heights through $H_+-H_-=\sqrt{1-A}$; in particular, the $\rho_\pm$ equilibria coalesce to the common value $\log(1/2)$ as $A\to1$.    At $\sigma=0$, \eqref{ss_profile_equation} reduces to the second order Li\'enard-type equation
\begin{equation}\label{ss_reduced_scalar_equation}
    \rho''=F(\rho)-G(\rho)\rho',
\end{equation}
and this special structure allows for the use of powerful planar dynamical systems techniques.  The following summarizes the bore wave existence result for \eqref{ss_reduced_scalar_equation} proved in \cite{Stevenson_Tice_bores} and introduces some notation that will be used throughout the paper.  See Figure \ref{parameter_regimes} for a graphical depiction of the sets $\mathcal{C}_{+1}$ and $\mathcal{C}_{-1}$ described in the theorem.

\begin{thm}\label{main_no_st_bores}
Let $a,\mu >0$ and $g \ge 0$.  For $A \in (0,1)$ and $F: \mathbb{R} \to \mathbb{R}$ given by \eqref{ss_coefficients} let $\rho_\ast = \rho_\ast(A) > \rho_+$ be uniquely determined by 
\begin{equation}\label{ss_turning_value}
    \int_{\rho_-}^{\rho_\ast}F(r)\,dr=0.
\end{equation}
Define the sets 
\begin{equation}\label{ss_parameter_regions}
    \begin{aligned}
        \mathcal C_{+1}&=\left\{(g,A)\in[0,\infty)\times(0,1): g<A^2e^{-3\rho_\ast(A)}\right\},\\
        \mathcal C_{-1} &=\left\{(g,A)\in[0,\infty)\times(0,1): g>A^2e^{-3\rho_-(A)}\right\}.
\end{aligned}
\end{equation}
Suppose $\iota\in\{-1,1\}$ and $(g,A)\in\mathcal C_{\iota}$.  Then  the vector field $X_0 : \mathbb{R}^2 \to \mathbb{R}^2$ defined by $X_0(x)=(x_2,F(x_1)-G(x_1)x_2)$ has a smooth heteroclinic solution $x_0:\mathbb{R}\to\mathbb{R}^2$ satisfying
\begin{equation}\label{ss_reduced_limits}
        \lim_{t\to\pm\infty}x_0(\iota t)=(\rho_\pm,0).
\end{equation}
The convergence is exponential, and the orbit $\Gamma_0=x_0(\mathbb{R})\cup\{(\rho_-,0),(\rho_+,0)\}$ is compact. The equilibrium point $(\rho_-,0)$ is a hyperbolic saddle, while the equilibrium point $(\rho_+,0)$ is a  hyperbolic sink if $\iota=1$ and a hyperbolic source if $\iota=-1$.   The corresponding height $H = \exp(e_1 \cdot x_0)$, which solves \eqref{ss_height_equation} with $\sigma=0$, is a surging bore if $\iota =-1$ and an ebbing bore if $\iota = +1$.
\end{thm}
\begin{proof}
These are proved in Lemmas 2.1--2.3 and Theorem 2.5 of \cite{Stevenson_Tice_bores}.
\end{proof} 

\begin{remark}
The defining feature of the regions  $\mathcal{C}_{+1}$ and $\mathcal{C}_{-1}$ is that they determine the sign of $G$ on $[\rho_-,\rho_\ast]$, with $G >0$ in $\mathcal{C}_{+1}$ and $G < 0$ in $\mathcal{C}_{-1}$.  For $\iota \in \{-1,+1\}$ we have that $\mathcal{C}_\iota \subset \{ (g,A) \in [0,\infty) \times (0,1) : \iota (8-g) >0\}$.  Moreover, when $\iota(8-g)>0$ we have that $(g,A) \in \mathcal{C}_\iota$ for all $A$ sufficiently close to unity.
\end{remark}

\begin{figure}[!h]
    \centering
    \scalebox{0.65}{\includegraphics{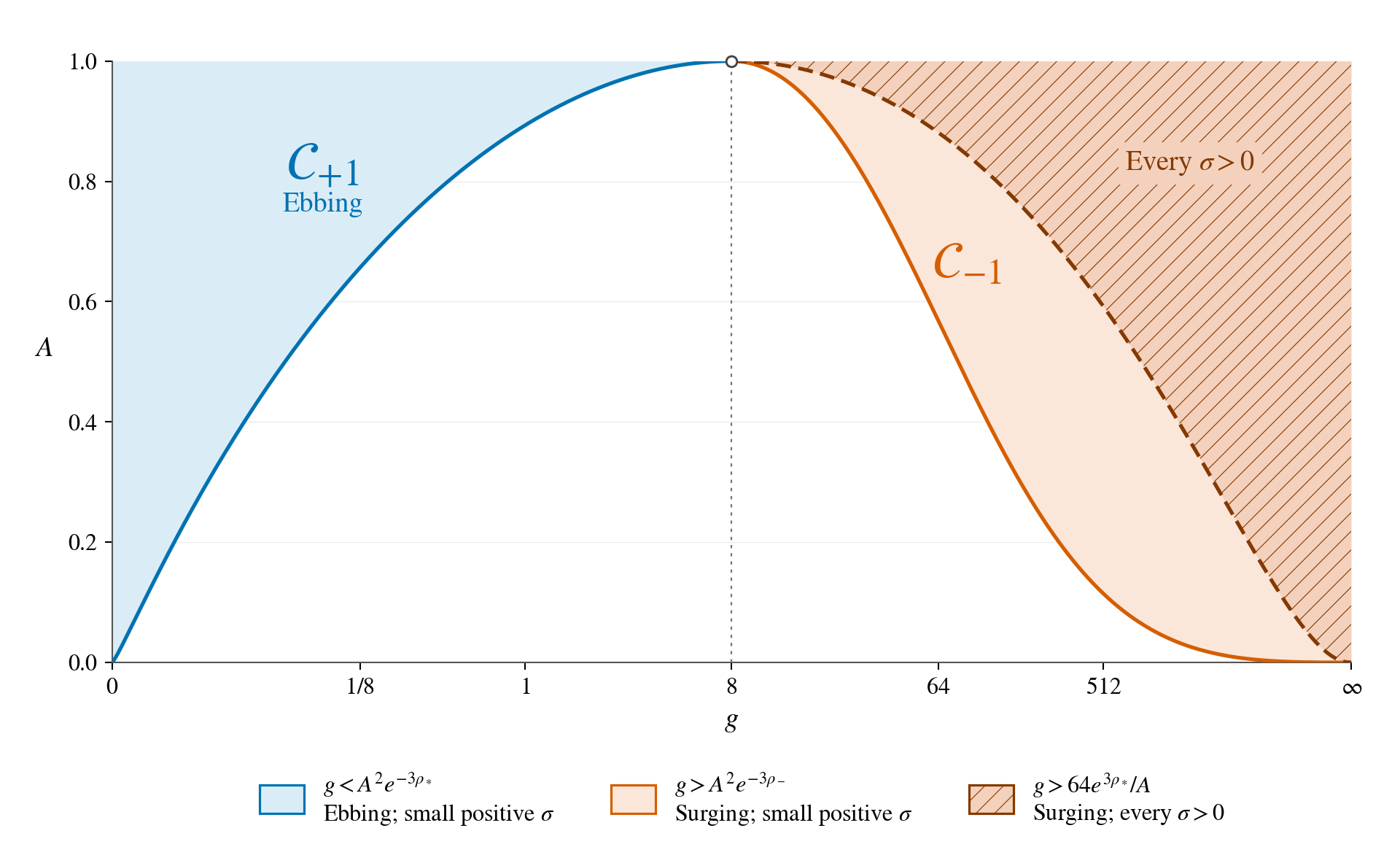}}
    \caption{Plots of the two parameter regimes $\mathcal{C}_{+1}$ and $\mathcal{C}_{-1}$.  The dashed subregion of $\mathcal{C}_{-1}$ is the valid parameter regime for Theorem \ref{main_large_grav}.
    }
    \label{parameter_regimes}
\end{figure}

When capillarity is retained, the third order problem \eqref{ss_profile_equation} is not amenable to analysis with planar dynamical systems techniques, and the existence of heteroclinic orbits (and thus bore wave solutions to \eqref{vsv_tw_form}) is a more subtle question.  Our main results are a trio of theorems that provide for the existence of such orbits under various assumptions.  Rather than planar dynamical systems techniques, we employ tools from geometric singular perturbation theory \cite{FENICHEL197953}, Leray-Schauder degree theory \cite{LeraySchauder1934}, and center manifold theory in order to construct the solutions.

Before stating our first theorem we clarify some notation.  For $k \in \mathbb{N}$ and $I \subseteq \mathbb{R}$ an interval, the Banach space $C^k_b(I)$ consists of all $k$-times continuously differentiable functions $f : I \to \mathbb{R}$ with all derivatives bounded, equipped with the norm 
\begin{equation}
 \|f\|_{C^k_b} = \sum_{j=0}^k \sup_{x \in I} |f^{(j)}(x)|.
\end{equation}
In particular, $C^0_b(I)$ is the space of bounded and continuous functions with the usual sup-norm.  We endow the space $C^\infty(\mathbb{R})$ of smooth real-valued functions on all of $\mathbb{R}$ with its usual Fr\'echet structure.  In particular, this means that $f_n \to f$ in $C^\infty(\mathbb{R})$ as $n \to \infty$ if and only if $f_n \to f$ in $C^k_b(I)$ as $n \to \infty$ for every bounded interval $I \subset \mathbb{R}$ and every $k \in \mathbb{N}$.  We also write $C^\infty_b(\mathbb{R}) = \bigcap_{k \in \mathbb{N}} C^k_b(\mathbb{R})$.

With this notation established we can now state our first main theorem, which shows that bore wave solutions persist for $(g,A) \in \mathcal{C}_{+1} \cup \mathcal{C}_{-1}$ provided that surface tension is sufficiently small.  We emphasize that the geometric singular perturbation theory tools used to prove the theorem establish the convergence of the solutions with $\sigma >0$ to the solutions with $\sigma =0$ provided by Theorem \ref{main_no_st_bores}.  We state the theorem here in terms of the height $H = e^\rho$, but its proof is carried out in terms of $\rho$.

\begin{thm}[Proved in Section \ref{sec_small_cap}]\label{main_small_sigma_persistence_in_C}
Fix $a,\mu>0$ and suppose that $(g,A)\in\mathcal C_{\iota}$  for $\iota\in\{-1,1\}$.  There exists $\sigma_\ast>0$ such that for every $0\le\sigma\le\sigma_\ast$ there exists a positive function $H_\sigma \in C^\infty_b(\mathbb{R})$ satisfying the following.
\begin{enumerate}
    \item The function $H_\sigma$ solves \eqref{ss_height_equation} for every $0\le\sigma\le\sigma_\ast$.

    \item For every $0\le\sigma\le\sigma_\ast$ we have that
\begin{equation}\label{ss_main_limits}
    \lim_{t\to\pm\infty} \bigl(H_\sigma(\iota t), H_\sigma'(\iota t),H_\sigma''(\iota t)\bigr)  =(H_\pm(A),0,0).
\end{equation}
In particular, $H_\sigma$ is an ebbing bore  flowing from $H_-(A)$ to $H_+(A)$ if $(g,A) \in \mathcal C_{+1}$, and a surging bore flowing from $H_+(A)$ to $H_-(A)$ if $(g,A) \in \mathcal C_{-1}$.

\item The  map $[0,\sigma_\ast]\ni\sigma \mapsto H_\sigma \in C^\infty(\mathbb{R})$ is continuous, and in particular  $H_\sigma\to H_0$ in $C^\infty(\mathbb{R})$ as $\sigma \to 0$.
\end{enumerate}
\end{thm}

Our second result identifies a subset of $\mathcal{C}_{-1}$ in which we can guarantee the existence of surging bore waves for all values of surface tension $\sigma >0$.  The proof is based on degree theory and so lacks some of the quantitative estimates of our other two main results.

\begin{thm}[Proved in Section \ref{sec_large_gravity}]\label{main_large_grav}
Fix $a,\mu >0$.  If $\sigma >0$, $A \in (0,1)$, and $g > 64 e^{3 \rho_\ast}/A$, then equation \eqref{ss_height_equation} has a positive surging bore wave solution $H\in C^\infty_b(\mathbb{R})$ satisfying
\begin{equation}\label{lg_final_limits}
        \lim_{t\to-\infty}(H(t),H'(t),H''(t))=(H_+,0,0),
        \text{ and }
        \lim_{t\to\infty}(H(t),H'(t),H''(t))=(H_-,0,0).
\end{equation}
\end{thm}

Our third main result shows that as long as $g \neq 8$, corresponding to a noncritical Froude number $\mathrm{Fr}=\sqrt{8/g}$, \eqref{ss_height_equation} admits bore wave solutions for every fixed surface tension $\sigma \ge 0$, provided that the flux parameter $A$ is taken sufficiently close to unity.  The center manifold techniques used to prove the theorem allow us to derive some more detailed information on the solutions than in the other two results. Indeed, the theorem shows 
that these solutions have small amplitude, are monotone, and have a common $\tanh$-type profile.  The latter profile only depends on $a$, $g$, and $A$, but we also derive higher order terms in the expansion that reveal the first contributions of viscosity and capillarity.  See Figure \ref{profile_sketch} for plots of these profiles.

\begin{thm}[Proved in Section \ref{sec_center_manifold}]\label{main_small_monotone_bores}
Fix $a,\mu>0$, $\sigma\ge0$, and $g\in [0,\infty) \backslash \{8\}$, and set $\iota=\operatorname{sgn}(8-g)$.  There exist $A_0\in(0,1)$, constants $C_0,C_1 >0$, and an open set $\mathcal{V} \subseteq \mathbb{R}^3$ containing $(1/2,0,0)$ such that for every $A\in(A_0,1)$ there exists a positive solution $H\in C^\infty_b(\mathbb{R})$ of \eqref{ss_height_equation} satisfying the following.

\begin{enumerate}
    \item We have $H(0)=1/2$, $\iota H'(t)>0$ for every $t\in\mathbb{R}$, and
\begin{equation}\label{cm_intro_limits}
    \lim_{t\to\pm\infty}\bigl(H(\iota t),H'(\iota t),H''(\iota t)\bigr)=(H_\pm,0,0).
\end{equation}
In particular $H$ is a monotone bore wave that is ebbing if $g<8$ and surging if $g>8$.

 \item The trajectory $(H,H',H'')$ lies in $\mathcal V$.  The solution $H$ is unique in the sense that any other positive heteroclinic solution $\widetilde H$ of \eqref{ss_height_equation} connecting $H_-$ and $H_+$ with $(\widetilde H,\widetilde H',\widetilde H'')(t) \in \mathcal{V}$ for all $t \in \mathbb{R}$ must satisfy $\widetilde H(t) = H(t-t_0)$ for some $t_0 \in \mathbb{R}$.

    \item The solution satisfies the uniform approximation
\begin{equation}\label{cm_intro_profile}
    \sup_{t\in\mathbb{R}}\left|H(t)-\frac12-\frac{(1-A)^{1/2}}2 \tanh\left(\frac{8a(1-A)^{1/2}t}{8-g}\right)\right|\le C_0(1-A).
\end{equation}
Moreover, for every $t\ge0$ we have that
\begin{equation}\label{cm_intro_tail}
    |H(\pm\iota t)-H_\pm|\le C_0(1-A)^{1/2} \exp(-C_1(1-A)^{1/2} t),
\end{equation}
and, for every $j\in\{1,2,3\}$ and $t\in\mathbb{R}$,
\begin{equation}\label{cm_intro_derivative_tails}
    |H^{(j)}(t)|\le C_0 (1-A)^{(j+1)/2} \exp(-C_1(1-A)^{1/2} |t|).
\end{equation}

\item The solution satisfies the refined approximation
\begin{multline}\label{cm_intro_second_order_profile}
    \sup_{t\in\mathbb R}\Bigg| H(t)-\frac12 -\frac{(1-A)^{1/2}}2
          \tanh\left(\frac{8a(1-A)^{1/2}t}{8-g}\right)
    \\
       -(1-A)\left(\frac{g+4}{8-g} +\frac{128a\mu}{(8-g)^2}\right)
          \operatorname{sech}^2\left(\frac{8a(1-A)^{1/2}t}{8-g}\right)
          \log\cosh\left(\frac{8a(1-A)^{1/2}t}{8-g}\right)
       \Bigg|\le C_0(1-A)^{3/2}.
\end{multline}

\item Let $H_0$ denote the solution with $\sigma =0$ and the same parameters $a,\mu,g,A$. Then
\begin{multline}\label{cm_intro_capillary_correction}
    \sup_{t\in\mathbb R}\Bigg|H(t)-H_0(t) -\frac{64\sigma a^2(1-A)^{3/2}}{(8-g)^3}
          \operatorname{sech}^2\left(\frac{8a(1-A)^{1/2}t}{8-g}\right)
    \\
       {}\times\left[ 3\tanh\left(\frac{8a(1-A)^{1/2}t}{8-g}\right) -\frac{16a(1-A)^{1/2}t}{8-g}\right]
       \Bigg|\le C_0(1-A)^2.
\end{multline}
\end{enumerate}
\end{thm}

\begin{figure}[!h]
    \centering
    \scalebox{0.75}{\includegraphics{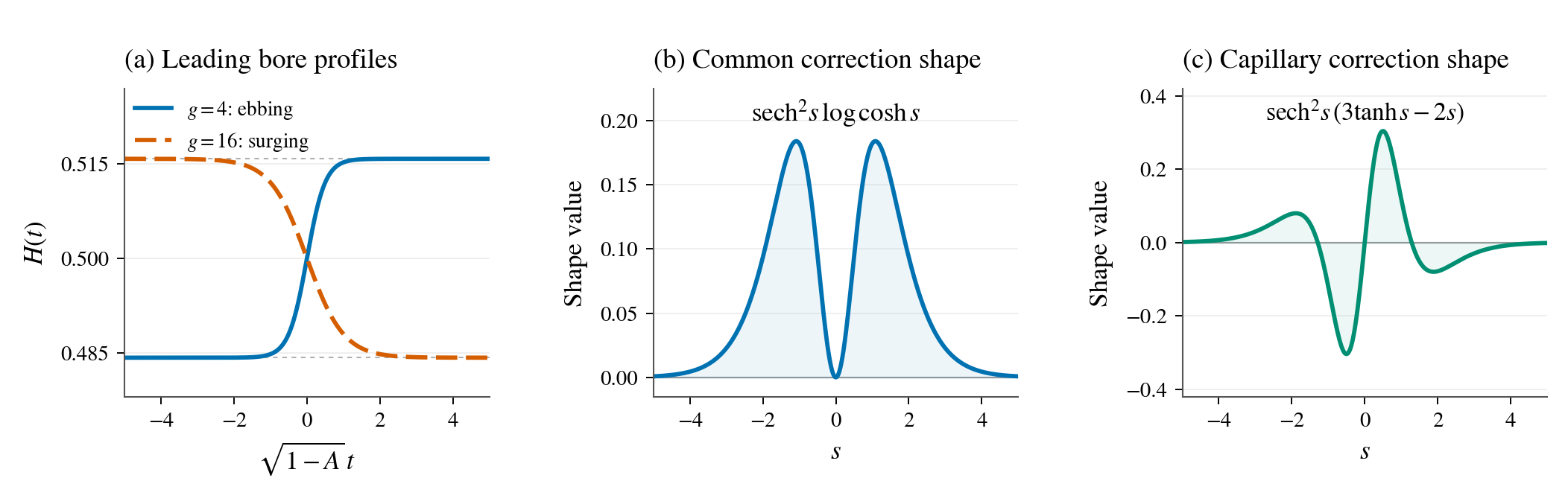}}
    \caption{Plots of approximate profiles identified in Theorem \ref{main_small_monotone_bores}.  The $\tanh$ profiles from item 3 are visualized in plot $(a)$.  The next order viscous corrections from item 4 is plotted in $(b)$.  The highest order correction from capillarity, identified in item 5, is plotted in $(c)$.
    }
    \label{profile_sketch}
\end{figure}

\begin{remark} 
It is possible to extend the estimate \eqref{cm_intro_derivative_tails} to be valid for all $j \ge 4$ provided the constant $C_0$ on the right is changed to a constant $C_j$, depending on $j$.  We have omitted these estimates for the sake of brevity.
\end{remark}

We conclude this introduction by mentioning a few related works in the literature. Hydraulic shock profiles for an inviscid version of  \eqref{vsv_tw_form} with quadratic bottom drag and no capillarity
were  studied by Yang and Zumbrun \cite{YangZumbrun2020} and Faye, Rodrigues, Yang, and Zumbrun
\cite{FayeRodriguesYangZumbrun2024}. The appearance of the leading order $\tanh$ profile in Theorem \ref{main_small_monotone_bores} is consistent with the classical theory of small-amplitude viscous shocks; see Majda and Pego \cite{MajdaPego1985} for the center-manifold construction and Plaza and Zumbrun \cite{PlazaZumbrun2004} for the profile estimates.  Small-amplitude periodic traveling waves, known as roll waves, were constructed by Stevenson \cite{StevensonRoll2025} for a two-dimensional counterpart of \eqref{vsv_tw_form} with $g,\sigma>0$; the paper also indicates that the construction can be adapted to produce one-dimensional periodic waves.

%%%%%%%%%%%%%%%%%%%%%%%%%%%%%%%%%%%%%%%%%%%%%%%%%%%%%%%%%%%%
%%%%%%%%%%%%%%%%%%%%%%%%%%%%%%%%%%%%%%%%%%%%%%%%%%%%%%%%%%%%
%%%%%%%%%%%%%%%%%%%%%%%%%%%%%%%%%%%%%%%%%%%%%%%%%%%%%%%%%%%%
\section{Small capillarity and $(g,A) \in \mathcal{C}_{+1} \cup \mathcal{C}_{-1}$}\label{sec_small_cap}
%%%%%%%%%%%%%%%%%%%%%%%%%%%%%%%%%%%%%%%%%%%%%%%%%%%%%%%%%%%%
%%%%%%%%%%%%%%%%%%%%%%%%%%%%%%%%%%%%%%%%%%%%%%%%%%%%%%%%%%%%
%%%%%%%%%%%%%%%%%%%%%%%%%%%%%%%%%%%%%%%%%%%%%%%%%%%%%%%%%%%%  
 
In this section our aim is to prove Theorem \ref{main_small_sigma_persistence_in_C}, which shows that bore wave solutions persist for all $(g,A) \in \mathcal{C}_{+1} \cup \mathcal{C}_{-1}$ for small surface tension.  We will prove this using the machinery of geometric singular perturbation theory, as initially developed by Fenichel \cite{FENICHEL197953}.  See also the survey of Jones \cite{Jones1995}.

To begin we recall some definitions.

\begin{definition}[Critical graphs and their reduced equations]
\label{ss_critical_graph_definition}
    Suppose that $p : \mathbb{R}^n \times \mathbb{R} \times \mathbb{R} \to \mathbb{R}^n$ and   $q : \mathbb{R}^n \times \mathbb{R} \times \mathbb{R} \to \mathbb{R}$ are smooth, and consider the differential system given by 
    \begin{equation}\label{ss_general_system}
    \begin{cases}
        x'= p(x,y,\varepsilon) \\ 
        \varepsilon y'=q(x,y,\varepsilon).     
\end{cases}
\end{equation}
The variable $x$ is called the slow variable, while $y$ is called the fast variable.    We define the following.  
\begin{enumerate}
 \item  Suppose that $V \subseteq \mathbb{R}^n$ is open and $h_0: V \to \mathbb{R}$ is smooth.  We say $h_0$ defines a critical graph for the system \eqref{ss_general_system} if  $q(x,h_0(x),0)=0$ for every $x\in V$.   We then say that $h_0$ induces the reduced differential equation 
\begin{equation}\label{ss_general_reduced_equation}
        x'=p(x,h_0(x),0)
\end{equation}
    over $V$.   We say that the graph of $h_0$ is normally hyperbolic over a compact set $K\subset V$ if $\partial_y q(x,h_0(x),0)\neq0$ for every $x\in K$.  We say it is normally attracting if this derivative is negative.
    
 \item Suppose $\varepsilon \neq 0$ and $h_\varepsilon : V_\varepsilon \to \mathbb{R}$ is a $C^1$ function on an open set $V_\varepsilon \subseteq \mathbb{R}^n$.  We say its graph $y = h_\varepsilon(x)$ is locally invariant for \eqref{ss_general_system} on $V_\varepsilon$  if the vector field $(p(\cdot,h_\varepsilon(\cdot),\varepsilon), \varepsilon^{-1} q(\cdot,h_\varepsilon(\cdot),\varepsilon))$ is tangent to its graph over $V_\varepsilon$.  It is a simple matter to check that this is equivalent to the identity
    \begin{equation}\label{ss_general_invariance_identity}
        \varepsilon \nabla h_\varepsilon(x)\cdot  p(x,h_\varepsilon(x),\varepsilon)  =q(x,h_\varepsilon(x),\varepsilon) \text{ for } x \in V_\varepsilon.
\end{equation}
    When this holds, solutions of the ODE $x'=p(x,h_\varepsilon(x),\varepsilon)$ lift to solutions
    $(x,h_\varepsilon(x))$ of \eqref{ss_general_system}.
\end{enumerate}
\end{definition}

The  next theorem is a simple scalar graphical version of the much more general first theorem of Fenichel \cite{FENICHEL197953}.  Our formulation is similar to those found in Section 1.2 of \cite{Jones1995}, Proposition 2.1 of \cite{EoEunKangOh2024}, and Theorem 2 of \cite{Hek2010}.

\begin{thm}\label{ss_graph_persistence}
In the setting of the first item of Definition \ref{ss_critical_graph_definition}, suppose that $B \subset V$ is a closed ball and that there exists $\beta>0$ such that 
\begin{equation}\label{ss_general_fast_bound}
        \partial_y q(x,h_0(x),0)\le-\beta     \text{ for all }x\in B.
\end{equation}
For each  $1 \le r \in \mathbb{N}$, there exist  $\varepsilon_0>0$, an open neighborhood $V_B$ of $B$, and a map $h\in C^r(V_B \times(-\varepsilon_0,\varepsilon_0))$ such that $h(x,0)=h_0(x)$, the graphs of $h_\varepsilon =h(\cdot,\varepsilon)$ over $V_B$ satisfy    \eqref{ss_general_invariance_identity} for all  $0<|\varepsilon|<\varepsilon_0$, and $h_\varepsilon \to h_0$ in $C^r(B)$ as $\varepsilon \to 0$.
\end{thm}
\begin{proof}
By continuity, we may select a closed ball $K$ such that $B \subset K^\circ \subset K \subset V$ and the estimate \eqref{ss_general_fast_bound} holds with $\beta$ replaced by $\beta/2$ and $B$ replaced by $K$.  We then  apply Theorem 9.1  of \cite{FENICHEL197953} to the fast vector field $(x,y)\mapsto(\varepsilon p(x,y,\varepsilon),q(x,y,\varepsilon))$ in the vicinity of the compact critical graph $\{(x,h_0(x)):x\in K\}$, which possesses a single normal eigenvalue $\partial_yq(x,h_0(x),0)\le-\beta/2$.  We then invoke the graph representation described in Section X of \cite{FENICHEL197953} to obtain a locally invariant $C^r$ graph  $y=h(x,\varepsilon)$, defined for $x\in K^\circ$ and sufficiently small $|\varepsilon|$, and satisfying  $h(x,0)=h_0(x)$.  We then set $V_B=K^\circ$ and note that local invariance yields \eqref{ss_general_invariance_identity}, while the $C^r$ regularity of $h$ and the compactness of $B$ prove the stated convergence.
\end{proof}

With Theorem \ref{ss_graph_persistence} in hand, we can now present the

\begin{proof}[Proof of Theorem~\ref{main_small_sigma_persistence_in_C}]

We divide the proof into five steps.

\emph{Step 1: Construction of an invariant graph.}  To begin, we define the smooth functions $h_0 : \mathbb{R}^2 \to \mathbb{R}$ and  $b : \mathbb{R} \to \mathbb{R}$ via $h_0(x)=F(x_1)-G(x_1)x_2$ and $b(x_1)=4\mu A e^{-2x_1}$, where we recall that $F,G$ are defined in \eqref{ss_coefficients}.   In turn, we define the smooth maps    $P : \mathbb{R}^2 \times \mathbb{R} \times \mathbb{R} \to \mathbb{R}^2$, and $Q : \mathbb{R}^2 \times \mathbb{R} \times \mathbb{R} \to \mathbb{R}$ via     
\begin{equation}
 P(x,y,\sigma) = (x_2,y) \text{ and } Q(x,y,\sigma)= b(x_1)\bigl(h_0(x)-y\bigr) -\sigma\bigl(x_2^3+3x_2y\bigr).
\end{equation}
Making the substitution  $x=(\rho,\rho')$ and $y=\rho''$, it is a simple matter to check that for $\sigma >0$ the ODE \eqref{ss_profile_equation} is equivalent to the system
\begin{equation}\label{ss_full_system}
        x'=P(x,y,\sigma) \text{ and } \sigma y'=Q(x,y,\sigma).
\end{equation}
Select a closed ball $B \subset \mathbb{R}^2$ such that $\Gamma_0 \subset B^\circ$, where $\Gamma_0$ is the compact set from  Theorem \ref{main_no_st_bores} that contains the $\sigma =0$ heteroclinic orbit.  Then the critical set over $B$ is the graph $y=h_0(x)$, and we have the estimate 
\begin{equation}\label{ss_fast_eigenvalue}
        \partial_yQ(x,h_0(x),0)=-b(x_1) \le -\min_{z\in B}b(z_1)<0 \text{ for all }x \in B.
\end{equation}
We may thus apply Theorem \ref{ss_graph_persistence} with $r=2$ to produce  $\sigma_0>0$, an open neighborhood $V_B \supset B$, and a map $h\in C^2(V_B \times(-\sigma_0,\sigma_0))$ such that $h(x,0)=h_0(x)$, the graphs of $h_\sigma =h(\cdot,\sigma)$ satisfy
\begin{equation}\label{ss_invariance_identity}
        \sigma \nabla h_\sigma(x)\cdot  P(x,h_\sigma(x),\sigma)  =Q(x,h_\sigma(x),\sigma) \text{ for } x \in V_B \text{ and } 0<|\sigma|<\sigma_0,
\end{equation}
and $h_\sigma \to h_0$ in $C^2(B)$ as $\sigma \to 0$.  In what follows we will employ the vector field $X_\sigma : V_B \to \mathbb{R}^2$ defined by
\begin{equation}\label{ss_perturbed_planar_field}
        X_\sigma(x)= P(x,h_\sigma(x),\sigma)=  (x_2,h_\sigma(x)),
\end{equation}
which is a  $C^2$ perturbation over $B$ of the vector field $X_0$ defined in Theorem \ref{main_no_st_bores}.

\emph{Step 2: Identification of the equilibria.}   Theorem \ref{main_no_st_bores} guarantees that the matrices  $DX_0(\rho_\pm,0)$ are invertible, so the implicit function theorem provides $0 < \sigma_1 \le \sigma_0$ and $C^2$ maps $(-\sigma_1,\sigma_1) \ni \sigma \mapsto z_{\pm,\sigma} \in \mathbb{R}^2$ such that $z_{\pm,\sigma}$ are the unique zeros of $X_\sigma$ in disjoint open neighborhoods of $(\rho_\pm,0)$.  Employing \eqref{ss_perturbed_planar_field} and \eqref{ss_invariance_identity} in conjunction then shows that $e_2 \cdot z_{\pm,\sigma}=0$, $h_\sigma(z_{\pm,\sigma}) =0$, and $0 = Q(z_{\pm,\sigma}, 0, \sigma) = b(e_1 \cdot z_{\pm,\sigma}) F(e_1 \cdot z_{\pm,\sigma})$.  Since $b>0$ and the unique zeros of $F$ are $\rho_\pm$, we deduce that 
\begin{equation}\label{ss_fixed_equilibria}
        z_{\pm,\sigma}=(\rho_\pm,0) \text{ and } h_\sigma(\rho_\pm,0)=0 \text{ for all } |\sigma| < \sigma_1.
\end{equation}
 By continuity, after decreasing $\sigma_1$ if necessary, both equilibria remain hyperbolic, with the same numbers of stable and unstable eigenvalues as for $X_0$.

\emph{Step 3: A continuous choice of the outgoing branch.} Next, we define the vector field $Y_\sigma = \iota X_\sigma$. With this renormalization, $Y_\sigma$ always has a saddle equilibrium at $(\rho_-,0)$ and a sink at $(\rho_+,0)$. Letting $x_0$ be the trajectory from Theorem \ref{main_no_st_bores} and setting $u_0(t)=x_0(\iota t)$, we find that $u_0$ is a heteroclinic orbit for the vector field $Y_0$ that goes from $(\rho_-,0)$ to $(\rho_+,0)$.

The parameter-dependent unstable manifold theorem (see, for instance, Theorems 9.4--9.6 of  \cite{Teschl2012}) provides $\delta>0$, $0<\sigma_2<\sigma_1$, an open neighborhood $U$ of $(\rho_-,0)$ with $\overline U\subset B^\circ$, and a $C^1$ map $\gamma:(-\delta,\delta)\times(-\sigma_2,\sigma_2)\to U$ such that $\gamma(\cdot,\sigma)$ is an injective parameterization of a local unstable curve and $\gamma(0,\sigma)=(\rho_-,0)$. The theorem also allows $\delta$, $\sigma_2$, and $U$ to be chosen so that the ODE trajectories starting on these curves are defined for all time $t \le 0$, remain in $U$, and converge to $(\rho_-,0)$ as $t \to -\infty$.

Translating $u_0$ in time if necessary, we may assume that  $u_0((-\infty,0]) \subset \gamma((-\delta,\delta),0)$.  This then allows us to choose  $\theta_0\in(-\delta,\delta)\setminus\{0\}$ such that  $u_0(0)=\gamma(\theta_0,0)$.  Using this, we then define
\begin{equation}
    z_\sigma=\gamma(\theta_0,\sigma)
\end{equation}
and note that, by construction,  $z_\sigma\neq(\rho_-,0)$, $z_\sigma$ depends continuously on $\sigma$, and $z_0=u_0(0)$. Let $u_\sigma$ be the solution of $u_\sigma'=Y_\sigma(u_\sigma)$ with $u_\sigma(0)=z_\sigma$. Then the conclusions of the unstable manifold theorem stated above guarantee that
\begin{equation}\label{ss_backward_limit}
    u_\sigma((-\infty,0] ) \subset U \text{ and } \lim_{t\to-\infty}u_\sigma(t)=(\rho_-,0).
\end{equation}

\emph{Step 4: Forward convergence to the sink.} Write $M=DY_0(\rho_+,0)$ and note that since all eigenvalues of $M$ have negative real parts, the integral
\begin{equation}
       S=\int_0^\infty \exp(sM^{\mathsf{T}}) \exp(sM) ds
\end{equation}
is well-defined and defines a  symmetric positive definite matrix.  Integrating the derivative of $e^{sM^{\mathsf{T}}}e^{sM}$ over $s \in (0,\infty)$ provides the identity
\begin{equation}\label{ss_lyapunov_matrix}
        M^{\mathsf{T}}S+SM=-I.
\end{equation}
Write   $Y_\sigma((\rho_+,0)+w)=Mw+R_\sigma(w)$ and note that since $Y_\sigma(\rho_+,0)=0$ and $DY_\sigma(x)\to M$ as $(x,\sigma)\to((\rho_+,0),0)$, we may choose $r_+>0$ and $0<\sigma_3<\sigma_2$ such  that $\overline{B((\rho_+,0), r_+)} \subset B^\circ$ and
    \begin{equation}\label{ss_sink_remainder}
        |R_\sigma(w)|\le\frac{|w|}{4\|S\|} \text{ for all } |w|\le r_+ \text{ and } |\sigma|<\sigma_3,
\end{equation}
where   $\|S\|$ denotes the Euclidean operator norm of $S$.  For a solution with $w=u_\sigma-(\rho_+,0) \in \overline{B(0, r_+)}$, \eqref{ss_lyapunov_matrix} and \eqref{ss_sink_remainder} imply that 
    \begin{equation}\label{ss_sink_lyapunov_decay}
            ( Sw \cdot w)' = -|w|^2+2 SR_\sigma(w) \cdot w  \le-\tfrac12|w|^2 \le-\frac{1}{2\|S\|}Sw\cdot w.
\end{equation}
Now choose $c>0$ small enough that 
    \begin{equation}
        E = \{(\rho_+,0)+w: Sw \cdot w \le c\} \subset B((\rho_+,0), r_+).
\end{equation}
Then \eqref{ss_sink_lyapunov_decay} shows that $E$ is positively invariant for every $|\sigma|<\sigma_3$, and that all solutions starting in $E$ exist for all positive time and converge to $(\rho_+,0)$ as $t \to \infty$.

Since $u_0(t)\to(\rho_+,0)$ as $t \to \infty$, we may choose $T>0$ with $u_0(T)\in E^\circ$.  The compact set $u_0([0,T])$ lies in $B^\circ$.  Then continuous dependence on the vector field and initial data, together with the convergence $z_\sigma\to z_0$ as $\sigma \to 0$, provides a $0<\sigma_4<\sigma_3$ such that if $|\sigma|<\sigma_4$ then the solution $u_\sigma$ exists on $[0,T]$ and satisfies $u_\sigma([0,T]) \subset B^\circ$ as well as $u_\sigma(T)\in E^\circ$.  It therefore extends to every $t\ge T$, stays in $E$ during this time, and satisfies
    \begin{equation}\label{ss_forward_limit}
        \lim_{t\to+\infty}u_\sigma(t)=(\rho_+,0).
\end{equation}
Combining this with results of Step 3 shows that the entire heteroclinic trajectory $u_\sigma$ is contained in $B^\circ$.

\emph{Step 5: Lifting to the profile equation and continuous dependence.} We now choose $0<\sigma_\ast<\sigma_4$ and restrict to $0\le\sigma\le\sigma_\ast$. Write $x_\sigma(t)=u_\sigma(\iota t)$ and  $\rho_\sigma(t)=e_1\cdot x_\sigma(t)$, and note that since $Y_\sigma=\iota X_\sigma$ and $\iota^2=1$, we have $x_\sigma'=X_\sigma(x_\sigma)$. In particular,
\begin{equation}\label{ss_profile_derivative_identities}
\rho_\sigma'=e_2\cdot x_\sigma \text{ and } \rho_\sigma''=h_\sigma(x_\sigma).
\end{equation}
For $\sigma>0$, the invariance identity \eqref{ss_invariance_identity} shows that $(x_\sigma,h_\sigma(x_\sigma))$ solves \eqref{ss_full_system} globally in time, and hence  $\rho_\sigma$ solves \eqref{ss_profile_equation}.  At $\sigma=0$, the identity $h(\cdot,0)=h_0$ gives \eqref{ss_reduced_scalar_equation}. Each $\rho_\sigma$ is smooth thanks to the ODE it satisfies.

Combining \eqref{ss_backward_limit}, \eqref{ss_forward_limit}, \eqref{ss_fixed_equilibria}, and \eqref{ss_profile_derivative_identities} shows that
\begin{equation}
        \lim_{s\to\pm\infty} \bigl(\rho_\sigma(\iota s), \rho_\sigma'(\iota s), \rho_\sigma''(\iota s)\bigr) =(\rho_\pm,0,0).
\end{equation}
The phase trajectory of each $\rho_\sigma$ has compact closure, and upon repeatedly differentiating the ODE to solve for higher order derivatives in terms of lower order ones we readily deduce that  $\rho_\sigma\in C_b^\infty(\mathbb{R})$.

Finally, the choice $z_\sigma=\gamma(\theta_0,\sigma)$ makes the initial data continuous in $\sigma$, and so continuous dependence for the ODE establishes the continuity of $\sigma\mapsto x_\sigma$ in $C^0_b(I)$ for every compact interval $I\subset\mathbb{R}$.  The two identities in \eqref{ss_profile_derivative_identities} and the continuity of $h$ then prove the continuity of $\sigma\mapsto\rho_\sigma$ in $C^2(I)$ for each such $I$, including at $\sigma=0$.  To upgrade this continuity, we rewrite \eqref{ss_profile_equation} as
\begin{equation}\label{cont_promote}
    \sigma\rho_\sigma'''+\alpha_\sigma\rho_\sigma''= b(\rho_\sigma)\bigl(F(\rho_\sigma)-G(\rho_\sigma)\rho_\sigma'\bigr)-\sigma(\rho_\sigma')^3, \text{ where } \alpha_\sigma=b(\rho_\sigma)+3\sigma\rho_\sigma'.
\end{equation}
Since $x_\sigma(\mathbb{R})\subset B$, decreasing $\sigma_\ast$ if necessary provides the lower bound $\alpha_\sigma(t) \ge \beta>0$ uniformly in $\sigma$ and $t\in \mathbb{R}$. Every bounded global solution of $\sigma v'+\alpha_\sigma v=f$ obeys the estimate $\|v\|_{C^0_b}\le\beta^{-1}\|f\|_{C^0_b}$; indeed, this is trivial when $\sigma =0$ and follows by multiplying by the integrating factor and integrating from $-\infty$ when $\sigma>0$. Differentiating \eqref{cont_promote} $n$ times gives an equation of this type for $v=\rho_\sigma^{(n+2)}$, with $f$ depending only on $\sigma$ and the derivatives of $\rho_\sigma$ through order $n+1$. The uniform boundedness of $\rho_\sigma,\rho_\sigma'$ and induction therefore show that
\begin{equation}
    \sup_{0\le\sigma\le\sigma_\ast}\|\rho_\sigma\|_{C_b^m(\mathbb{R})}<\infty    \text{ for every }m\in\mathbb{N}.
\end{equation}
These bounds, together with the Arzel\`a-Ascoli theorem and the established $C^2_b(I)$ continuity, then  imply continuity in $C^m_b(I)$ for every $m\in\mathbb{N}$ and every compact interval $I\subset \mathbb{R}$.  Hence, the map $\sigma\mapsto\rho_\sigma$ is continuous in $C^\infty(\mathbb{R})$.

Now set $H_\sigma=e^{\rho_\sigma}$. The equivalence of \eqref{ss_height_equation} and \eqref{ss_profile_equation}, including at $\sigma=0$, shows that $H_\sigma$ solves \eqref{ss_height_equation} for every $0\le\sigma\le\sigma_\ast$. The chain rule gives $H_\sigma\in C_b^\infty(\mathbb{R})$ and the limits \eqref{ss_main_limits}. Finally, composition with the exponential function is continuous on $C^\infty(\mathbb{R})$, so the continuity of $\sigma\mapsto\rho_\sigma$ proves the third item.
\end{proof}

%%%%%%%%%%%%%%%%%%%%%%%%%%%%%%%%%%%%%%%%%%%%%%%%%%%%%%%%%%%%
%%%%%%%%%%%%%%%%%%%%%%%%%%%%%%%%%%%%%%%%%%%%%%%%%%%%%%%%%%%%
%%%%%%%%%%%%%%%%%%%%%%%%%%%%%%%%%%%%%%%%%%%%%%%%%%%%%%%%%%%%
\section{Arbitrary capillarity and large gravity} \label{sec_large_gravity}
%%%%%%%%%%%%%%%%%%%%%%%%%%%%%%%%%%%%%%%%%%%%%%%%%%%%%%%%%%%%
%%%%%%%%%%%%%%%%%%%%%%%%%%%%%%%%%%%%%%%%%%%%%%%%%%%%%%%%%%%%
%%%%%%%%%%%%%%%%%%%%%%%%%%%%%%%%%%%%%%%%%%%%%%%%%%%%%%%%%%%%

Our goal in this section is to prove Theorem \ref{main_large_grav}, which shows that surging bore solutions exist for all surface tension coefficients $\sigma >0$, provided that gravity is sufficiently large.  Our proof adapts the techniques of Bonheure, Cid, De Coster, and Sanchez  \cite{bonheure}, combining a priori estimates and  Leray-Schauder degree theory.     The results of \cite{bonheure} concern equations of the form $u'''=f(u)+p(t)u'$ with prescribed nonnegative $p$, whereas \eqref{ss_profile_equation} contains additional terms involving the second derivative and nonlinear products of derivatives.  Our main task is then to adapt their methods to account for these terms.  We refer to the survey of Mawhin \cite{Mawhin1999} for an overview of the Leray-Schauder degree.

Since $\sigma >0$ we can rewrite \eqref{ss_profile_equation} in the equivalent forms 
\begin{equation}\label{lg_profile_equation}
    \rho'''=\frac{4\mu A}{\sigma}e^{-2\rho} \bigl(F(\rho)-G(\rho)\rho'-\rho''\bigr) -(\rho')^3-3\rho'\rho''
\end{equation}
and 
\begin{equation}\label{lg_equation_height_form}
    \rho''=F(\rho)-G(\rho)\rho' - \frac{\sigma}{4\mu A} e^\rho(e^\rho)''',
\end{equation}
both of which we will use in our analysis.  We begin by identifying the bounded nonconstant solutions.

\begin{thm}\label{bnded_is_hetero_orbit}
    Suppose that $\rho\in C_b^3(\mathbb{R})$ is a nonconstant solution of \eqref{lg_profile_equation}, and that there is a constant $c<0$ such that $G(\rho(t))\le c$ for every $t\in\mathbb{R}$. Then $\rho$ is a heteroclinic orbit satisfying
    \begin{equation}\label{lg_bounded_orbit_limits}
        \lim_{t\to-\infty}(\rho(t),\rho'(t),\rho''(t))=(\rho_+,0,0),
        \text{ and }
        \lim_{t\to\infty}(\rho(t),\rho'(t),\rho''(t))=(\rho_-,0,0).
\end{equation}
\end{thm}
\begin{proof}
Define the smooth potential $V: \mathbb{R} \to \mathbb{R}$ associated to $F$ and the total energy $\mathcal{E} : \mathbb{R} \to \mathbb{R}$ of $\rho$ via 
\begin{equation}\label{lg_potential_energy}
        V(r)=-\int_{\rho_-}^rF(s)\,ds
        \text{ and }  \mathcal{E}(t)=V(\rho(t))+\frac{(\rho'(t))^2}{2} + \frac{\sigma}{4\mu A}(e^\rho)'(t)(e^\rho)''(t).
\end{equation}
Multiplying \eqref{lg_equation_height_form} by $\rho'$ provides the identity
\begin{equation}\label{lg_energy_calculation}
        \rho''\rho'=F(\rho)\rho'-G(\rho)(\rho')^2 -\frac{\sigma}{4\mu A}(e^\rho)'(e^\rho)''',
\end{equation}
from which we then deduce that 
\begin{equation}\label{lg_energy_identity}
     \mathcal{E}'(t)= -G(\rho(t))(\rho'(t))^2 +\frac{\sigma}{4\mu A}\bigl((e^\rho)''(t)\bigr)^2\ge0.
\end{equation}
In particular, since $\rho\in C_b^3(\mathbb{R})$  we have the bound 
\begin{equation}\label{lg_integrated_energy}
\sup_{T >0} \left[    \int_{-T}^T-G(\rho)(\rho')^2\,dt +\frac{\sigma}{4\mu A}\int_{-T}^T\bigl((e^\rho)''\bigr)^2\,dt \right] = \sup_{T >0} \left[\mathcal{E}(T)-\mathcal{E}(-T) \right] < \infty,
\end{equation}
and thus $\rho'\in L^2(\mathbb{R})$ and $(e^\rho)''\in L^2(\mathbb{R})$.  Again since $\rho\in C_b^3(\mathbb{R})$, we know that both $\rho'$ and $(e^\rho)''$ are uniformly continuous; it follows from this and the $L^2$ inclusions that $\rho'(t)\to0$ and $(e^\rho)''(t)\to0$ as $t\to\pm\infty$. Since $(e^\rho)''=e^\rho\bigl((\rho')^2+\rho''\bigr)$  and $\rho$ is bounded, we also must have that $\rho''(t)\to 0$ as $t \to \pm \infty$.  Next we appeal to the ODE \eqref{lg_profile_equation} to express  $\rho'''$ as a smooth function of the bounded, Lipschitz functions $\rho,\rho',\rho''$, which then allows us to deduce that $\rho'''$ is  uniformly continuous.  

Next, we recall an elementary analysis fact.  Suppose that a function $\varphi : (0,\infty) \to \mathbb{R}$ is differentiable with $\varphi'$ uniformly continuous.  If $\varphi(t) \to 0$ as $t \to \infty$, then $\varphi'(t) \to 0$ as $t \to \infty$ as well.  Indeed, if not then we can pick $\varepsilon >0$, $s \in \{-1,+1\}$,  and an increasing sequence $0 < t_n \to \infty$ such that $s \varphi'(t_n) \ge \varepsilon$ for all $n$.  Uniform continuity then grants a $\delta >0$ such that $s \varphi'(t) \ge \varepsilon/2$ for $t \in (t_n-\delta,t_n+\delta)$, which then contradicts the assumption that $\varphi(t) \to 0$ as $t \to \infty$.

Using the previous fact with the information that  $\rho''(t) \to 0$ as $t \to \pm \infty$ and $\rho'''$ is uniformly continuous, we deduce that  $\rho'''(t) \to 0$ as $t\to \pm \infty$ as well.  Returning to \eqref{lg_profile_equation} and sending $t \to \pm \infty$  then shows that $F(\rho(t))\to 0$ as $t \to \pm\infty$.  Since $F$ has only the distinct zeros $\rho_\pm$, we deduce from this that $\rho$ itself converges at $\pm \infty$ with 
\begin{equation}
 \lim_{t \to \pm \infty} \rho(t) \in \{\rho_-,\rho_+\}, \text{ and hence } \lim_{t \to \pm \infty} \mathcal{E}(t) = V\left( \lim_{t \to \pm \infty} \rho(t) \right).
\end{equation}
Since $\rho$ is not constant and $-G(\rho)\ge-c>0$, we may integrate the bound \eqref{lg_energy_identity} to deduce that 
\begin{equation}\label{lg_strict_energy_order}
\lim_{t\to \infty} \mathcal{E}(t) - \lim_{t \to -\infty} \mathcal{E}(t) >0, \text{ and hence }
 V\left( \lim_{t \to - \infty} \rho(t) \right)
        <  V\left( \lim_{t \to  \infty} \rho(t) \right).
\end{equation}
Since  $F>0$ on $(\rho_-,\rho_+)$, we have $V(\rho_+)<V(\rho_-)$, and this completes the proof of \eqref{lg_bounded_orbit_limits}.
\end{proof}

Next, we prove a simple lemma that introduces some parameters that will be important in the subsequent analysis.  It will be useful to recall that $\rho_\ast$ is defined by \eqref{ss_turning_value}.

\begin{lem}\label{lg_amplitude_barrier_lemma}
Set 
\begin{equation}\label{lg_midpoint}
    \rho_a=\frac{\rho_++\rho_-}{2} \in (\rho_-,\rho_+)
\end{equation}
and 
\begin{equation}\label{lg_amplitude_barrier_formulas}
    \alpha=2\rho_a-\rho_\ast=\log(A/4)-\rho_\ast.
\end{equation}
Then $\alpha < \rho_- < \rho_a < \rho_+ < \rho_\ast$, and $\alpha$ and $\rho_\ast$ are uniquely determined by the identities 
\begin{equation}\label{lg_amplitude_barriers}
    \int_{\rho_a}^{\rho_\ast}F(s)\,ds = - \int_{\rho_-}^{\rho_a} F(s)\,ds,
    \text{ and }
    \int_{\alpha}^{\rho_a} F(s)\,ds = - \int_{\rho_a}^{\rho_+}F(s)\,ds.
\end{equation}
Moreover, $G<0$ on $[\alpha,\rho_\ast]$ if and only if $g>A^2e^{-3\alpha}=64 e^{3 \rho_\ast}/A$.
\end{lem}
\begin{proof}
The identities in \eqref{lg_amplitude_barriers} are equivalent to
\begin{equation}
    \int_{\rho_-}^{\rho_\ast}F(s)\,ds=0   \text{ and }    \int_{\alpha}^{\rho_+}F(s)\,ds=0.
\end{equation}
The former is what uniquely defines $\rho_\ast$ in  \eqref{ss_turning_value}, so we only need to study the latter.

We have that  $e^{\rho_-}e^{\rho_+}=A/4$ and $2\rho_a=\log(A/4)$, and hence
\begin{equation}\label{lg_reflection_identity}
    F(2\rho_a-r)=F(r) \text{ for every }r\in\mathbb R.
\end{equation}
Since $\rho_\ast>\rho_+$, we have $\alpha<2\rho_a-\rho_+=\rho_-$. Changing variables using \eqref{lg_reflection_identity} yields
\begin{equation}
    \int_{\alpha}^{\rho_+}F(s)\,ds=\int_{\rho_-}^{\rho_\ast}F(s)\,ds=0,
\end{equation}
which shows that $\alpha$ indeed satisfies the desired integral identity.   Uniqueness follows since $F<0$ on $(-\infty,\rho_-)$, so the map $r\mapsto\int_r^{\rho_+}F(s)\,ds$ is strictly increasing there.  Finally, the assertions about $G$ follow from direct calculation.
\end{proof}

The next lemma constructs a crucial continuous affine map $K$.

\begin{lem}\label{lg_affine_inverse}
Let $1 \le n \in \mathbb{N}$.  The operator $K:C^0([-n,n])\to C^3([-n,n])$ defined by $Kh = \rho$ for $\rho$  satisfying $\rho'''=h$, $\rho'(\pm n)=0$, and $\rho(0)=\rho_a$ is well-defined, continuous, and affine. More precisely, $K(h)=\rho_a+K_T h$, where $K_T:C^0([-n,n])\to C^3([-n,n])$ is bounded and linear. Regarded as a map into $C^2([-n,n])$, $K$ maps bounded sets to relatively compact sets.
\end{lem}
\begin{proof}
Given $h \in C^0([-n,n])$ write  
\begin{equation}\label{lg_triple_primitive}
        J_h(t)=\int_0^t\int_0^s\int_0^r h(q)\,dq\,dr\,ds.
\end{equation}
Every solution of $\rho'''=h$ has the form $\rho(t)=J_h(t)+c_2t^2+c_1t+c_0$.  The condition at $0$ implies $c_0=\rho_a$, while the $\pm n$ conditions require
\begin{equation}\label{lg_inverse_coefficients}
        c_2=\frac{J_h'(-n)-J_h'(n)}{4n} \text{ and } c_1=-\frac{J_h'(n)+J_h'(-n)}{2}.
\end{equation}
Thus,
\begin{equation}\label{lg_inverse_formula}
        K(h)(t)=J_h(t)+ \frac{J_h'(-n)-J_h'(n)}{4n}t^2 -\frac{J_h'(n)+J_h'(-n)}{2}t+\rho_a,
\end{equation}
which proves both existence and uniqueness.  This also  shows that $K_T=K-\rho_a$ is linear and that
\begin{equation}\label{lg_inverse_bound}
        \|K_T h\|_{C^3_b([-n,n])}\le C_n\|h\|_{C^0_b([-n,n])}
\end{equation}
for a constant $C_n < \infty$.  Hence, $K$ is continuous and affine. The last assertion follows from the compact inclusion $C^3([-n,n])\subset C^2([-n,n])$.
\end{proof}

Our next lemma builds an important nonlinear map related to part of the ODE \eqref{lg_profile_equation}.

\begin{lem}\label{lg_nonlinear_derivative}
The operator $\varphi:C^2([-n,n])\to C^0([-n,n])$ defined by
\begin{equation}\label{lg_phi}
        \varphi(u)=-\frac{4\mu A}{\sigma}e^{-2u}u'' -(u')^3-3u'u''
\end{equation}
is smooth, with derivative given by 
\begin{equation}\label{lg_phi_derivative}
 D\varphi(\rho)h = \frac{8\mu A}{\sigma}e^{-2\rho}\rho''h -\frac{4\mu A}{\sigma}e^{-2\rho}h''
-3(\rho')^2h'-3h'\rho''-3\rho'h''.
\end{equation}
\end{lem}
\begin{proof}
Since $C^0([-n,n])$ is a commutative Banach algebra, multiplication is continuous and multilinear and hence smooth. The maps $u\mapsto u'$ and $u\mapsto u''$ are bounded linear maps from $C^2([-n,n])$ to $C^0([-n,n])$.  It remains to check that the exponential composition map is smooth and compute its derivatives, as once this is established the identity \eqref{lg_phi_derivative} follows from this and the standard formulas for the derivatives of multilinear maps.

Up to norm equivalence, the space $C^2([-n,n])$ is also a commutative Banach algebra with unit, so the map $C^2([-n,n]) \ni u \mapsto \exp(u) = \sum_{k=0}^\infty u^k /k! \in C^2([-n,n])$ is well-defined, smooth, and satisfies $D \exp(u) h = \exp(u) h$.  On the other hand, since norm convergence in this space implies pointwise convergence, we have $\exp(u)(x) = e^{u(x)}$ for all $u \in C^2([-n,n])$ and $x \in [-n,n]$.  Thus, the map $C^2([-n,n]) \ni u \mapsto e^{-2u} \in C^2([-n,n]) \hookrightarrow C^0([-n,n])$ is smooth, and its derivative at $\rho \in C^2([-n,n])$ is given by $h\mapsto-2e^{-2\rho}h$.
\end{proof}

The next lemma establishes some essential a priori estimates.

\begin{lem}\label{lg_derivative_bounds}
Assume that $g>64e^{3\rho_\ast}/A$. There exist constants $K_1,K_2>0$  such that if  $\rho\in C^3([-n,n])$ for $1 \le n \in \mathbb{N}$  and $\lambda \in [0,1]$ satisfy
 \begin{equation}\label{lg_homotopy_equation}
        \rho''=\lambda F(\rho)-\lambda G(\rho)\rho' -\frac{\sigma}{4\mu A} e^\rho(e^\rho)''',
        \quad \rho'(\pm n)=0, 
        \text{ and } \alpha\le\rho\le\rho_\ast,
\end{equation}
then 
\begin{equation}\label{lg_uniform_derivative_bounds}
        \|\rho'\|_{C^0_b([-n,n])}\le K_1 \text{ and } \|\rho''\|_{C^0_b([-n,n])}\le K_2.
\end{equation}
\end{lem}
\begin{proof}
Fix $n,\lambda,\rho$ as in the hypotheses, write $Y=e^\rho$, and define the quantities
\begin{equation}\label{lg_estimate_constants}
D=  \max_{[\alpha,\rho_\ast]}V-\min_{[\alpha,\rho_\ast]}V \text{ and } k=\min_{[\alpha,\rho_\ast]}(-G)>0,
\end{equation}
noting that the final inequality is a consequence of Lemma \ref{lg_amplitude_barrier_lemma}.  Multiplying \eqref{lg_homotopy_equation} by $\rho'$, integrating by parts over $[-n,n]$, and using the boundary conditions shows that 
\begin{equation}\label{lg_finite_interval_energy}
   \frac{\sigma}{4\mu A}\int_{-n}^n(Y'')^2\,dt  + \lambda\int_{-n}^n[-G(\rho)](\rho')^2\,dt
        =\lambda\bigl[V(\rho(n))-V(\rho(-n))\bigr]  \le\lambda D.
\end{equation}
If $\lambda=0$, this identity implies that $Y''=0$, but then the boundary conditions $Y'(\pm n)=0$ mandate that $\rho$ is constant.   On the other hand, if $\lambda>0$ then the identity implies that 
\begin{equation}\label{lg_l2_estimates}
   \|Y''\|_{L^2([-n,n])}\le\sqrt{4\mu A D/\sigma} \text{ and } \|\rho'\|_{L^2([-n,n])}\le\sqrt{D/k}.
\end{equation}

Let $\tilde{Y}$ denote the extension of  $Y$ outside $[-n,n]$ by $Y(n)$ in $(n,\infty)$ and $Y(-n)$ in $(-\infty,-n)$.  The fact that $Y'(\pm n)=0$ implies that  $\tilde{Y}$ is continuously differentiable and has weak second derivative with $\tilde{Y}'' = Y''$ in $(-n,n)$ and $0$ otherwise.  In particular, $\|\tilde{Y}\|_{C^0_b(\mathbb{R})}\le e^{\rho_\ast}$ and $\|\tilde{Y}''\|_{L^2(\mathbb{R})}   =\|Y''\|_{L^2([-n,n])}$.  We may then apply  Gagliardo-Nirenberg interpolation (see, for instance, Chapter 12 of \cite{GiovanniLeoni}) to bound
\begin{equation}\label{lg_interpolation}
  \|\tilde{Y}'\|_{C^0_b(\mathbb{R})} \le C \|\tilde{Y}''\|_{L^2(\mathbb{R})}^{2/3}
        \|\tilde{Y}\|_{C^0_b(\mathbb{R})}^{1/3} = C \|Y''\|_{L^2([-n,n])}^{2/3}
        \|Y\|_{C^0_b([-n,n])}^{1/3} 
\end{equation}
for some universal constant $C \ge 0$.  Since $\rho'=e^{-\rho}Y'$ and $\rho\ge\alpha$, combining \eqref{lg_l2_estimates} and \eqref{lg_interpolation} gives
\begin{equation}\label{lg_first_derivative_constant}
    \|\rho'\|_{C^0_b([-n,n])} \le e^{-\alpha}\|\tilde{Y}'\|_{C^0_b(\mathbb{R})}
    \le C e^{-\alpha+\rho_\ast/3}( 4\mu A D/\sigma)^{1/3} =:K_1.
\end{equation}
Moreover, $\rho''=e^{-\rho}Y''-(\rho')^2$, so
\begin{equation}\label{lg_second_derivative_l2}
 \|\rho''\|_{L^2([-n,n])} \le e^{-\alpha}\sqrt{4\mu A D/\sigma} +K_1\sqrt{D/k}=:Q.
\end{equation}
Rewriting the ODE as
\begin{equation}\label{lg_second_derivative_ode}
\rho'''=-(\frac{4\mu A}{\sigma}e^{-2\rho}+3\rho')\rho'' + \lambda\frac{4\mu A}{\sigma}e^{-2\rho} \bigl(F(\rho)-G(\rho)\rho'\bigr)-(\rho')^3
\end{equation}
then allows us to bound  $|\rho'''|\le B|\rho''|+L$ for
\begin{equation}\label{lg_gronwall_constants}
 B=\frac{4\mu A}{\sigma}e^{-2\alpha}+3K_1 \text{ and } 
 L=\frac{4\mu A}{\sigma}e^{-2\alpha} \bigl(\|F\|_{C^0_b([\alpha,\rho_\ast])} +K_1\|G\|_{C^0_b([\alpha,\rho_\ast])}\bigr)+K_1^3.
\end{equation}
For each $t\in[-n,n]$ choose an interval $I_t \subseteq [-n,n]$ of length one with $t \in I_t$.  By \eqref{lg_second_derivative_l2}, there exists $s \in I_t$  with $|\rho''(s)|\le Q$.  Applying Gronwall's inequality between $s$ and $t$ then yields the bound $|\rho''(t)|\le e^B(Q+L)$. This completes the proof upon setting  $K_2=e^B(Q+L)$.
\end{proof}

We now have all of the tools needed to prove the main theorem of this section, from which Theorem \ref{main_large_grav} will then readily follow.

\begin{thm}\label{lg_bounded_solution}
For every $g > 64 e^{3 \rho_\ast}/A$, equation \eqref{lg_profile_equation} has a nonconstant solution $\rho\in C_b^3(\mathbb{R})$ satisfying $\rho(0)=\rho_a$ and $\alpha\le\rho(t)\le\rho_\ast$ for every $t\in\mathbb{R}$.  
\end{thm}
\begin{proof}
We divide the proof into four steps.

\textit{Step 1: Setting up the modified equations and the homotopy.} Define the maps $F_\pm,  G^\ast : \mathbb{R} \to \mathbb{R}$ and $F^\ast : \mathbb{R}^2 \to \mathbb{R}$ via 
\begin{equation}\label{lg_truncated_F}
 F_-(r)=
\begin{cases}
      0&r<\rho_-,\\
      F(r)&\rho_-\le r\le\rho_\ast,\\
      F(\rho_\ast)&r>\rho_\ast,
\end{cases}
        \qquad
F_+(r)=
\begin{cases}
   F(\alpha)&r<\alpha,\\
   F(r)&\alpha\le r\le\rho_+,\\
   0&r>\rho_+,
\end{cases}
\end{equation}
 and
\begin{equation}\label{lg_truncated_FG}
G^\ast(r)=
\begin{cases}
   G(\alpha)&r<\alpha,\\
   G(r)&\alpha\le r\le\rho_\ast,\\
   G(\rho_\ast)&r>\rho_\ast,
\end{cases}
  \qquad
F^\ast(t,r)=
\begin{cases}
F_+(r)&t\ge0,\\
F_-(r)&t<0.
\end{cases}
\end{equation}
The functions $F_\pm$ and $G^\ast$ are Lipschitz, and $G^\ast<0$ on $\mathbb{R}$. For the sake of brevity, we write $X_n = C^2([-n,n])$ in what follows.  We also define  
\begin{multline}\label{lg_degree_domain}
\Omega_n = \{u\in X_n:\ \rho_-<u(0)<\rho_+,\; \alpha<u(t)<\rho_\ast \text{ for all }t \in [-n,n],\\  \|u'\|_{C^0_b} < 1 + K_1, \text{ and } \|u''\|_{C^0_b} < 1 + K_2\},
\end{multline}
which is  a bounded open subset of $X_n$ containing the constant function $\rho_a$. 

We propose to define the map $H:\overline{\Omega_n}\times[0,1]\to X_n$ via
\begin{equation}\label{lg_compact_homotopy}
H(u,\lambda)=K\left(\frac{4\mu A}{\sigma}e^{-2u}\bigl(\lambda F^\ast(\cdot,u)-\lambda G^\ast(u)u'-u''\bigr)
  -(u')^3-3u'u''\right),
\end{equation}
where $K$ is from Lemma \ref{lg_affine_inverse}.  To check that this is well-defined, fix $(u,\lambda)\in\overline{\Omega_n}\times[0,1]$.  Since $X_n \ni u \mapsto u(0)$ is continuous, we have $u(0)\in[\rho_-,\rho_+]$ and since  $F_+ =F_-=F$ on this interval we may compute 
\begin{equation}
\begin{aligned}
        \lim_{t\to0^-}F^\ast(t,u(t))&=F_-(u(0))=F(u(0))  \text{ and }  \\
        \lim_{t\to0^+}F^\ast(t,u(t))&=F_+(u(0))=F(u(0))=F^\ast(0,u(0)).
\end{aligned}
\end{equation}
Thus $t\mapsto F^\ast(t,u(t))$ is continuous at $0$, but continuity away from $0$ follows from the continuity of $u$ and $F_\pm$.  Since $u\in C^2([-n,n])$ and $G^\ast$ is continuous, every other term in the argument of $K$ is also continuous. Lemma \ref{lg_affine_inverse} then shows that $H(u,\lambda)\in C^3([-n,n])\subset X_n$, completing the proof that $H$ is well-defined.

The Lipschitz bounds for $F_\pm$ provide the bound
\begin{equation}\label{lg_truncated_composition_continuity}
 \|F^\ast(\cdot,u)-F^\ast(\cdot,v)\|_{C^0_b} \le 
 \max\{\operatorname{Lip}F_-,\operatorname{Lip}F_+\} \|u-v\|_{C^0_b}.
\end{equation}
A similar bound for $G^\ast$ holds, which together with the previous bound and Lemma \ref{lg_affine_inverse} proves continuity of $H$.  Its image is bounded in $C^3([-n,n])$, and hence relatively compact in $X_n$ by Arzel\`a-Ascoli.  Any fixed point of $H(\cdot,\lambda)$ must belong to $C^3([-n,n])$ and satisfy $u(0)=\rho_a$, $u'(\pm n)=0$.

\textit{Step 2: Proving nonzero Leray--Schauder degree.}   When $\lambda=0$ we have that $H(u,0)=K(\varphi(u))$, so Lemmas \ref{lg_affine_inverse} and \ref{lg_nonlinear_derivative} imply that $H(\cdot,0)$ naturally extends to a $C^1$ map on all of $X_n$.   We claim that the only fixed point of $H(\cdot,0)$ is the constant map $u=\rho_a$.  Indeed, a fixed point satisfies $u(0)=\rho_a$, $u'(\pm n)=0$, and $u''=-\frac{\sigma}{4\mu A} e^u(e^u)'''$.  Multiplying this  by $u'$ yields
\begin{equation}\label{lg_zero_homotopy_energy}
        \left[\frac{(u')^2}{2} + \frac{\sigma}{4\mu A}(e^u)'(e^u)''\right]'
        =\frac{\sigma}{4\mu A}\bigl((e^u)''\bigr)^2,
\end{equation}
and the term in brackets on the left vanishes at $\pm n$.  Thus, upon integrating over $[-n,n]$, we find that $(e^u)''=0$, but then the boundary condition $u'(\pm n)=0$ implies that $(e^u)'=0$.  Thus $u$ is constant and $u(0)=\rho_a$, so  $u=\rho_a$.  This proves the claim, which in particular tells us that there is no fixed point of $H(\cdot,0)$ on $\partial\Omega_n$ when $\lambda=0$.

Now define $T(u)=u-H(u,0)= u - K(\varphi(u))$. Since $DK=K_T$, Lemma \ref{lg_nonlinear_derivative} allows us to compute
\begin{equation}\label{lg_linearized_operator}
        DT(\rho_a)h=h+c_aK_T(h'')
        \text{ for } c_a=\frac{4\mu A}{\sigma}e^{-2\rho_a}>0.
\end{equation}
The map $h\mapsto K_T(h'')$ is compact on $X_n$, so $DT(\rho_a)$ is Fredholm. If $DT(\rho_a) h=0$, then $h\in C^3([-n,n])$ and satisfies  $h'''=- c_a h''$, $h(0)=0$, and $h'(\pm n)=0$.
Multiplying the ODE by $h'$ and integrating shows that 
\begin{equation}\label{lg_kernel_energy}
        -\int_{-n}^n(h'')^2\,dt =- c_a\int_{-n}^nh''h'\,dt=0,
\end{equation}
from which we deduce that $h''=0$.  In turn, the conditions at $\pm n$ imply that $h'=0$, but then $h(0)=0$ requires that $h=0$.  We have thus proved that  $DT(\rho_a)$ is injective, so the Fredholm alternative implies that  $DT(\rho_a)$ is invertible.  Since $\rho_a\in\Omega_n$ is the only zero of $T$ in $\overline{\Omega_n}$, the degree formula for nondegenerate zeros from Section 2 of \cite{Mawhin1999} yields
\begin{equation}\label{lg_degree_at_zero}
    \deg(T,\Omega_n,0)\in\{-1,1\}.
\end{equation}

\textit{Step 3: A priori bounds.} We now aim to show that  $H(\cdot,\lambda)$ has no fixed point on $\partial\Omega_n$ for any $\lambda\in[0,1]$.  The case $\lambda=0$ was proved in Step 2, so suppose that $\lambda>0$ and $\rho=H(\rho,\lambda)$ with $\rho\in\overline{\Omega_n}$.  Define the potentials $V_\pm : \mathbb{R} \to \mathbb{R}$ via 
\begin{equation}\label{lg_truncated_potentials}
        V_+(r)=\int_{\rho_a}^rF_+(s)\,ds \text{ and } V_-(r)=\int_{\rho_a}^rF_-(s)\,ds.
\end{equation}
Similarly, define $E_- : [-n,0]  \to \mathbb{R}$ and $E_+ : [0,n]  \to \mathbb{R}$ via 
\begin{equation}\label{lg_truncated_energies}
        E_\pm(t)=\lambda V_\pm(\rho(t))
            -\frac{(\rho'(t))^2}{2}
            -\frac{\sigma}{4\mu A}(e^\rho)'(t)(e^\rho)''(t).
\end{equation}

The identity $\rho = H(\rho,\lambda)$ requires that 
\begin{equation}\label{lg_truncated_energy_derivative}
E_\pm'(t)=\lambda G^\ast(\rho(t))(\rho'(t))^2 -\frac{\sigma}{4\mu A}\bigl((e^\rho)''(t)\bigr)^2                    \le 0.
\end{equation}
We claim that $E_-$ is strictly decreasing on $[-n,0]$ and $E_+$ is strictly decreasing on $[0,n]$. To prove the first assertion, suppose that $-n\le s<t\le0$ and $E_-(s)=E_-(t)$. Integrating \eqref{lg_truncated_energy_derivative} over $[s,t]$ and using $\lambda>0$ and $G^\ast<0$ shows that $\rho'=0$ in $[s,t]$ and hence $\rho=c$ there for some constant $c$.  Plugging into the ODE satisfied by $\rho$ on $[-n,0]$ shows that $F_-(c)=0$.  Hence, the constant function $c$ and $\rho$ solve the same autonomous ODE and agree in $(s,t)$.  Since $F_-$ and $G^\ast$ are locally Lipschitz, ODE uniqueness  implies that $\rho=c$ throughout $[-n,0]$, but then $\rho_a = \rho(0)=c$, which contradicts the identity $F_-(c)=0$.  This proves the claim for $E_-$, but a similar argument on $[0,n]$ proves the claim for $E_+$.

Set $m=V_+(\rho_-)=V_-(\rho_-)$ and $M=V_+(\rho_+)=V_-(\rho_+)$. The definitions of $F_\pm$ and the identities \eqref{lg_amplitude_barriers} give the bounds
\begin{equation}\label{lg_potential_barriers}
\begin{cases}
 V_+(r)\ge m \text{ and } V_-(r)\le M &\text{for } r\in\mathbb{R},\\
 V_+(r)\ge M &\text{for } r\le\alpha \text{ or } r\ge\rho_+,\\
 V_-(r)\le m &\text{for } r\le\rho_- \text{ or } r\ge\rho_\ast.
\end{cases}
\end{equation}
We will use these bounds and the strict decrease of $E_\pm$ to control the values of $\rho$ at its extrema.

Since $\rho(0)=\rho_a$ and $V_+(\rho_a)=V_-(\rho_a)=0$, we have the equality $E_-(0)=E_+(0)$, and we denote this common value by $E_0$.  The conditions $\rho'(\pm n)=0$ imply that both derivative terms in \eqref{lg_truncated_energies} vanish at $\pm n$. Consequently, the strict decrease of $E_\pm$ and \eqref{lg_potential_barriers} show that
\begin{equation}\label{lg_endpoint_energy_comparison}
    \lambda m\le\lambda V_+(\rho(n))<E_0<\lambda V_-(\rho(-n))\le\lambda M.
\end{equation}

Suppose first that $t\in(0,n]$ and $\rho'(t)=0$. Then $(e^\rho)'(t)=0$ as well, so \eqref{lg_truncated_energies} and \eqref{lg_endpoint_energy_comparison} imply that
\begin{equation}\label{lg_positive_critical_point}
    \lambda V_+(\rho(t))=E_+(t)<E_0<\lambda M.
\end{equation}
Thus $V_+(\rho(t))<M$, and \eqref{lg_potential_barriers} requires that $\alpha<\rho(t)<\rho_+$. Similarly, if $t\in[-n,0)$ and $\rho'(t)=0$, then
\begin{equation}\label{lg_negative_critical_point}
    \lambda V_-(\rho(t))=E_-(t)>E_0>\lambda m.
\end{equation}
The inequality $V_-(\rho(t))>m$ and \eqref{lg_potential_barriers} then require that $\rho_-<\rho(t)<\rho_\ast$.

Now consider the maximum and minimum of $\rho$ on $[0,n]$.  If either is attained at an interior point, then $\rho'$ vanishes there. The same is true at $n$ by the boundary condition, so the preceding argument places all of these values in $(\alpha,\rho_+)$.  At the remaining endpoint, $\rho(0)=\rho_a$ also belongs to this interval.  Hence, $\rho([0,n]) \subset (\alpha,\rho_+)$.  Arguing similarly on $[-n,0]$ and combining, we deduce that
\begin{equation}\label{lg_strict_range_bounds}
    \alpha<\rho(t)<\rho_+ \text{ for } t\in[0,n]
    \text{ and }
    \rho_-<\rho(t)<\rho_\ast \text{ for } t\in[-n,0].
\end{equation}

The bounds \eqref{lg_strict_range_bounds} ensure that $F^\ast(t,\rho(t))=F(\rho(t))$ and $G^\ast(\rho(t))=G(\rho(t))$ for all  $t \in [-n,n]$.  This means that $\rho$ actually satisfies \eqref{lg_homotopy_equation}, so Lemma \ref{lg_derivative_bounds} then provides the estimates
\begin{equation}\label{lg_strict_derivative_bounds}
    \|\rho'\|_{C^0_b([-n,n])}\le K_1<K_1+1
    \text{ and }
    \|\rho''\|_{C^0_b([-n,n])}\le K_2<K_2+1.
\end{equation}
Together with $\rho(0)=\rho_a$ and \eqref{lg_strict_range_bounds}, these show that $\rho\in\Omega_n$.  In summary, we have now shown that $H(\cdot,\lambda)$ has no fixed point on $\partial\Omega_n$ for any $\lambda\in[0,1]$.  Homotopy invariance and \eqref{lg_degree_at_zero} then allow us to compute
\begin{equation}\label{lg_degree_at_one}
    \deg(I-H(\cdot,1),\Omega_n,0)=\deg(T,\Omega_n,0)\neq0.
\end{equation}
The existence property of the degree (see, for instance, Theorem 2.1(ii) of  \cite{Mawhin1999}) then yields a fixed point of $H(\cdot,1)$ for every $n\ge 1$.

\textit{Step 4: Proving boundedness and the final construction.}  For $n \ge 1$ let $\rho_n$ denote the fixed point from Step 3, which solves \eqref{lg_profile_equation} on $[-n,n]$, satisfies $\rho_n(0)=\rho_a$ and $\rho_n'(\pm n)=0$, and obeys the bound
\begin{equation}\label{lg_exhaustion_bounds}
 \sup_{n\ge1}\|\rho_n\|_{C^2_b([-n,n])}\le\max\{|\alpha|,|\rho_\ast|\}+K_1+K_2.
\end{equation}
The bound and the ODE then show that $\sup_{n \ge 1} \|\rho_n'''\|_{C^0_b([-n,n])} < \infty$ as well.

On each fixed interval $[-j,j]$, discard the finitely many indices $n<j$.  The previous bounds show that the remaining functions and their first two derivatives are uniformly bounded and equicontinuous on $[-j,j]$.  Applying Arzel\`a-Ascoli on successive intervals and extracting a diagonal subsequence then provides a single subsequence satisfying 
\begin{equation}\label{lg_diagonal_convergence}
        \rho_{n_k}\to \rho
        \text{ in }C^2_b(I) \text{ as } k \to \infty \text{ for every compact interval }I \subset \mathbb{R}.
\end{equation}
Write the right-hand side of \eqref{lg_profile_equation} as $\Phi(\rho,\rho',\rho'')$, where $\Phi$ is smooth.  Then \eqref{lg_diagonal_convergence} implies that $\rho_{n_k}'''\to \Phi(\rho,\rho',\rho'')$ locally uniformly on $\mathbb{R}$.  Passing to the limit in the integral identity for $\rho_{n_k}''$ proves that $\rho\in C^3(\mathbb{R})$ and solves \eqref{lg_profile_equation}. The global uniform bounds pass to the limit, so $\rho\in C_b^3(\mathbb{R})$. Moreover, $\rho(0)=\rho_a$ and $\alpha\le\rho(t)\le\rho_\ast$ for all $t$. Since $F(\rho_a)>0$, the constant $\rho_a$ is not a solution; thus the limiting solution is nonconstant.
\end{proof}

We are now ready to present the

\begin{proof}[Proof of Theorem \ref{main_large_grav}]
Theorem \ref{lg_bounded_solution} provides a nonconstant solution $\rho\in C_b^3(\mathbb{R})$ with $\rho(0)=\rho_a$ and values in $[\alpha,\rho_\ast]$.  Since $G<0$ on this compact interval, we have that $G(\rho(t))\le\max_{r\in[\alpha,\rho_\ast]}G(r)<0$ for all $t\in\mathbb{R}$. Theorem \ref{bnded_is_hetero_orbit} then proves \eqref{lg_bounded_orbit_limits}.  Smoothness follows from a standard bootstrap argument based on the ODE \eqref{lg_profile_equation}.  As in Step 5 of the proof of Theorem \ref{main_small_sigma_persistence_in_C}, repeated differentiation of the smooth first order formulation of the ODE shows that $\rho\in C_b^\infty(\mathbb{R})$.
Setting $H=e^\rho$ now gives a positive solution of \eqref{ss_height_equation} in $C_b^\infty(\mathbb{R})$, and the chain rule yields \eqref{lg_final_limits}.
\end{proof}

%%%%%%%%%%%%%%%%%%%%%%%%%%%%%%%%%%%%%%%%%%%%%%%%%%%%%%%%%%%%
%%%%%%%%%%%%%%%%%%%%%%%%%%%%%%%%%%%%%%%%%%%%%%%%%%%%%%%%%%%%
%%%%%%%%%%%%%%%%%%%%%%%%%%%%%%%%%%%%%%%%%%%%%%%%%%%%%%%%%%%%
\section{Arbitrary surface tension and large flux} \label{sec_center_manifold}
%%%%%%%%%%%%%%%%%%%%%%%%%%%%%%%%%%%%%%%%%%%%%%%%%%%%%%%%%%%%
%%%%%%%%%%%%%%%%%%%%%%%%%%%%%%%%%%%%%%%%%%%%%%%%%%%%%%%%%%%%
%%%%%%%%%%%%%%%%%%%%%%%%%%%%%%%%%%%%%%%%%%%%%%%%%%%%%%%%%%%%

In this final section we aim to prove Theorem \ref{main_small_monotone_bores}, which shows that for every fixed $\sigma\ge0$ and $g\in[0,\infty)\backslash\{8\}$ there exist small amplitude, monotone bores of ebbing ($g<8$) or surging ($g>8$) type, provided that the flux parameter $A$ is close enough to unity. We prove the theorem by using center manifold techniques. Throughout the preliminary reduction we assume that $a,\mu,\sigma>0$ and $g\in[0,\infty)\backslash\{8\}$; the case $\sigma=0$ will be treated directly in the proof of Theorem \ref{cm_main_theorem}. For the reduction it will be convenient to return to the height $H=e^\rho$ and write \eqref{ss_height_equation} in the equivalent form
\begin{equation}\label{cm_height_equation}
4\mu A\left(\frac{H'}{H}\right)' =4a\left(1-H-\frac{A}{4H}\right) -\left(\frac{A^2}{H^2}-gH\right)H'  -\sigma HH'''.
\end{equation}
Recalling \eqref{ss_equilibrium_values}, the two positive constant solutions of \eqref{cm_height_equation} are 
\begin{equation}\label{cm_end_states}
    H_\pm(A)=e^{\rho_\pm(A)}=\frac{1\pm\sqrt{1-A}}{2}.
\end{equation}
Both go to $1/2$ as $A \to 1$, and $H=1/2$ is a solution to \eqref{cm_height_equation} when $A=1$.  This is the key to unlocking center manifold tools.

Our first lemma begins to check the hypotheses of the center manifold theorem.

\begin{lem}\label{cm_linearization}
The following are equivalent. 
\begin{enumerate}
 \item $H : \mathbb{R} \to (0,1)$ is $C^3$ and solves \eqref{cm_height_equation} with $A \in [0,1]$.
 \item $h = H-1/2 : \mathbb{R} \to (-1/2,1/2)$ is $C^3$ and $x = (h,h',h'')$ solves 
\begin{equation}\label{cm_extended_system}
x' = \Psi(x,\delta) 
\end{equation}
for $\delta = 1-A$ and $\Psi(x,\delta) = (x_2,x_3,\Psi_3(x,\delta))$ with 
\begin{multline}
   \Psi_3(x,\delta) =  \frac{1}{\sigma}\left[
     -\frac{4\mu(1-\delta)}{(1/2+x_1)^2}x_3+\frac{4\mu(1-\delta)}{(1/2+x_1)^3}x_2^2 \right.
       \\ 
       \left. -\left(\frac{(1-\delta)^2}{(1/2+x_1)^3}-g\right)x_2+\frac{4a}{(1/2+x_1)^2}\left(\frac{\delta}{4}-x_1^2\right) \right]. 
\end{multline}
\end{enumerate}
We have that $\Psi(0,0)=0$, so $0$ is an equilibrium of the ODE \eqref{cm_extended_system} with $\delta=0$.  Moreover,  $L = D_x \Psi(0,0) \in \mathbb{R}^{3 \times 3}$ has characteristic polynomial
\begin{equation}\label{cm_characteristic}
    \det(\lambda I-L)=\lambda\left(\lambda^2+\frac{16\mu}{\sigma}\lambda+\frac{8-g}{\sigma}\right).
\end{equation}
The zero eigenvalue is simple, with center space  given by
\begin{equation}\label{cm_center_space}
    E^c=\{(h,0,0):h\in\mathbb{R}\}.
\end{equation}
The two remaining eigenvalues have strictly negative real parts when $g<8$, and are real with opposite signs when $g>8$.
\end{lem}
\begin{proof}
The identities
\begin{equation}
    H-H^2-\frac{A}{4} = \frac{\delta}{4}-h^2 \text{ and }     \left(\frac{H'}{H}\right)'=\frac{H''}{H}-\frac{(H')^2}{H^2}
\end{equation}
transform \eqref{cm_height_equation} into
\begin{equation}\label{cm_height_third_derivative}
        \sigma H'''  =-\frac{4\mu A}{H^2}H''+\frac{4\mu A}{H^3}(H')^2-\left(\frac{A^2}{H^3}-g\right)H'+\frac{4a}{H^2}\left(\frac{\delta}{4}-h^2\right).
\end{equation}
Substitution of $h=H-1/2$ gives \eqref{cm_extended_system}, and these manipulations are reversible when $H>0$.

We then compute 
\begin{equation}\label{cm_linear_matrix}
    L = D_x\Psi(0,0)=
    \begin{pmatrix}
        0&1&0\\
        0&0&1\\
        0&-(8-g)/\sigma&-16\mu/\sigma
    \end{pmatrix},
\end{equation}
which yields \eqref{cm_characteristic} and \eqref{cm_center_space}.  If $g<8$, the two nonzero roots have negative sum and positive product, so they are either both negative or a conjugate pair with negative real part. If $g>8$, their product is negative, so they are real and have opposite signs.
\end{proof}

The next lemma invokes the center manifold theorem to construct locally invariant curves for \eqref{cm_extended_system} and to derive a reduced-form scalar ODE on them.

\begin{lem}\label{cm_reduction}
There exist $\delta_0>0$, open sets $0 \in  \mathcal{O} \subseteq \mathbb{R}^3$ and  $0 \in D \subseteq \mathbb{R}^2$, and functions $P,Q\in C^4(D)$ such that the following hold.
\begin{enumerate}
\item For every $|\delta|<\delta_0$ the curve
\begin{equation}\label{cm_graph}
    \mathcal{M}_\delta=\{(h,P(h,\delta),Q(h,\delta)):(h,\delta)\in D\}
\end{equation}
is locally invariant for \eqref{cm_extended_system}.
\item Every global solution $x : \mathbb{R} \to \mathcal{O}$ to \eqref{cm_extended_system} with fixed $|\delta|<\delta_0$ satisfies $x(\mathbb{R}) \subseteq \mathcal{M}_\delta$. 
 \item We have the identities 
\begin{equation}\label{cm_graph_tangent}
        P(0,0)=Q(0,0)=0, \quad P_h(0,0)=Q_h(0,0)= Q_\delta(0,0) =0, \text{ and } P_\delta(0,0)=\frac{4a}{8-g}.
\end{equation}
\item There exist an open set $0 \in D_0 \subset D$ and a function $k\in C^3(D_0)$ such that
\begin{equation}\label{cm_exact_factorization}
    P(h,\delta)=k(h,\delta)\left(\frac\delta4-h^2\right) \text{ and } k(0,0)=\frac{16a}{8-g}.
\end{equation}
For each fixed $|\delta|<\delta_0$, the flow of \eqref{cm_extended_system} on the portion of $\mathcal{M}_\delta$ corresponding to $(h,\delta)\in D_0$ is described by 
\begin{equation}\label{cm_scalar_system}
    h'=k(h,\delta)\left(\frac\delta4-h^2\right).
\end{equation}
Every solution of \eqref{cm_scalar_system} with $\mathbb{R} \ni t \mapsto (h(t),\delta)\in D_0$ lifts to a solution 
\begin{equation}
\mathbb{R} \ni t \mapsto (h(t),P(h(t),\delta),Q(h(t),\delta))\in \mathcal{M}_\delta 
\end{equation}
of \eqref{cm_extended_system}.
\end{enumerate}

\end{lem}
\begin{proof}

Let $\Psi: (-1/2,\infty) \times \mathbb{R}^3 \to \mathbb{R}^3$ be the smooth map from Lemma \ref{cm_linearization}.  We augment this to form $\tilde{\Psi}(x,\delta) = (\Psi(x,\delta),0) \in  \mathbb{R}^4$.  In the usual way, this allows us to study the $\delta$-dependent family of ODEs $x' = \Psi(x,\delta)$ as the new autonomous system $(x,\delta)' = \tilde{\Psi}(x,\delta)$.  Its linearization at the origin is
\begin{equation}
    \tilde L=D\tilde{\Psi}(0,0)=
    \begin{pmatrix}
        0&1&0&0\\
        0&0&1&0\\
        0&-(8-g)/\sigma&-16\mu/\sigma&4a/\sigma\\
        0&0&0&0
    \end{pmatrix}.
\end{equation}
Since $g\ne8$, its generalized center space is
\begin{equation}
    \tilde E^c=\ker\tilde L^2  =\left\{\left(h,\frac{4a}{8-g}\delta,0,\delta\right):(h,\delta)\in\mathbb{R}^2\right\},
\end{equation}
and its remaining two eigenvalues have nonzero real parts by Lemma \ref{cm_linearization}.

The nonlinear remainder $(x,\delta) \mapsto\tilde{\Psi}(x,\delta)-\tilde L(x,\delta)$ is smooth and vanishes to first order at $(0,0)$.  We may therefore invoke the center manifold theorem; specifically, we use Theorem 4.21(i)--(ii) of \cite{MagalRuan2009} with $k=4$. This provides a two-dimensional  locally invariant $C^4$ manifold $\tilde{\mathcal M}$ containing $(0,0)$ and tangent there to $\tilde E^c$, and an open neighborhood $\mathcal{W}$ of $(0,0)$ such that every complete trajectory of the augmented system contained in $\mathcal{W}$ lies within $\tilde{\mathcal{M}}$.

The projection $\pi(x,\delta)=(x_1,\delta)$ restricts to a linear isomorphism from $\tilde{E}^c$ onto $\mathbb{R}^2$. The inverse function theorem applied to $\pi|_{\tilde{\mathcal M}}$ then provides an open neighborhood $D$ of $(0,0)$ and functions $P,Q\in C^4(D)$ such that, after restricting the manifold to a neighborhood of the origin,
\begin{equation}
    \tilde{\mathcal M}=\{(h,P(h,\delta),Q(h,\delta),\delta):(h,\delta)\in D\}.
\end{equation}
Since $\delta'=0$, its sections at fixed $\delta$ give the locally invariant curves $\mathcal M_\delta$ in \eqref{cm_graph} for every sufficiently small $\delta$. This also gives the asserted joint $C^4$ regularity of $P$ and $Q$.

Shrink $\mathcal W$ so that every point of $\tilde{\mathcal M}$ in $\mathcal W$ belongs to the portion covered by this graph representation. Choose an open neighborhood $\mathcal O$ of the origin in $\mathbb{R}^3$ and $\delta_0>0$ sufficiently small that the preceding construction applies for every $|\delta|<\delta_0$ and $\mathcal O\times(-\delta_0,\delta_0)\subset\mathcal W$. If $x:\mathbb{R}\to\mathcal O$ solves $x'=\Psi(x,\delta)$ for a fixed $|\delta|<\delta_0$, then $(x(t),\delta)$ is a complete trajectory of the augmented system contained in $\mathcal W$. It therefore lies in $\tilde{\mathcal M}$, which means that $x(t)\in\mathcal M_\delta$ for every $t\in\mathbb{R}$. This proves the first two items.  The passage of $\tilde{\mathcal M}$ through the origin gives $P(0,0)=Q(0,0)=0$, and tangency to $\tilde E^c$ gives
\begin{equation}
    P_h(0,0)=Q_h(0,0)=Q_\delta(0,0)=0
    \text{ and }P_\delta(0,0)=\frac{4a}{8-g}.
\end{equation}
This proves \eqref{cm_graph_tangent} and the third item.

Next, we use local invariance to see that 
\begin{equation}\label{cm_graph_invariance}
        Q(h,\delta)=P_h(h,\delta)P(h,\delta) \text{ and }
        Q_h(h,\delta)P(h,\delta)= \Psi_3(h,P(h,\delta),Q(h,\delta),\delta).
\end{equation}
The nearby equilibria of \eqref{cm_extended_system} are exactly the states and parameter values satisfying $x_2=x_3 =0$ and $\delta=4x_1^2$. For all sufficiently small $h$ we have the inclusion  $(h,0,0) \in  \mathcal O$ and the bound $4h^2<\delta_0$.  Its constant trajectory therefore lies in $\mathcal{M}_{4h^2}$, which implies that $P(h,4h^2)=Q(h,4h^2)=0$. Now choose $D_0\subset D$ small enough that this identity holds for every $h$ occurring in $D_0$ and the segment joining $(h,4h^2)$ to $(h,\delta)$ lies in $D$ whenever $(h,\delta)\in D_0$.  The fundamental theorem of calculus then allows us to compute
\begin{equation}\label{cm_division_formula}
        P(h,\delta)=(\delta-4h^2)\int_0^1 P_\delta\bigl(h,4h^2+s(\delta-4h^2)\bigr)\,ds 
\end{equation}
and thus define $k$ via 
\begin{equation}
        k(h,\delta)=4\int_0^1 P_\delta\bigl(h,4h^2+s(\delta-4h^2)\bigr)\,ds.
\end{equation}
Since $P\in C^4(D)$, we have $k\in C^3(D_0)$, and \eqref{cm_graph_tangent} guarantees that $k(0,0)=4P_\delta(0,0)=16a/(8-g)$, completing the proof of  \eqref{cm_exact_factorization}, which together with the form of $\Psi$ then implies \eqref{cm_scalar_system}.  The final assertion of the fourth item follows from \eqref{cm_graph_invariance}.
\end{proof}

In the next lemma we study the reduced equation derived in the previous lemma.

\begin{lem}\label{cm_scalar_lemma} 
Let $0 \in U \subseteq \mathbb{R}^2$ be an open set, let $\omega \in C^3(U)$ be such that  $\omega_0:=\omega(0,0)\neq 0$, and set $\iota=\operatorname{sgn}(\omega_0)$.  There exists $\eta_0 >0$ with $[-\eta_0/2,\eta_0/2] \times [0,\eta_0^2] \subset U$ such that for every $0 < \eta < \eta_0$ there exists $h_\eta \in C^3_b(\mathbb{R})$ satisfying the following.

\begin{enumerate}
 \item We have that 
\begin{equation}\label{cm_scalar_limits}
     -\eta/2 < h_\eta(t) < \eta/2 \text{ and } \iota h_\eta'(t)>0 \text{ for }t \in \mathbb{R}, \text{ and }  \lim_{t\to\pm\infty}h_\eta(t)=\pm\frac{\iota\eta}{2}.
\end{equation}

 \item $h_\eta$ is the unique solution to the IVP
\begin{equation}\label{cm_scalar_eta}
\begin{cases}
 h_\eta'=\omega(h_\eta,\eta^2)\left(\frac{\eta^2}{4}-h_\eta^2\right) \\
 h_\eta(0) = 0.
\end{cases}
\end{equation}

 \item There are constants $C,c>0$, independent of $\eta$, such that we have the estimates 
\begin{equation}\label{cm_scalar_approximation}
        \sup_{t\in\mathbb{R}} \left|h_\eta(t)-\frac\eta 2 \tanh\left(\frac{\omega_0\eta t}{2}\right)\right|
        \le C\eta^2,
\end{equation}
and
\begin{equation}\label{cm_scalar_tail_bounds}
\left|h_\eta(t)-\frac{\iota\eta\operatorname{sgn}(t)}{2}\right|  + \sum_{j=1}^3 \eta^{-j} |h_\eta^{(j)}(t)| \le C\eta e^{-c\eta|t|}  \text{ for }t\in\mathbb{R}.
\end{equation}
\end{enumerate}
\end{lem}
\begin{proof} 

We divide the proof into three steps.

\emph{Step 1: Construction of the solution.} Since $\omega_0\neq0$, continuity allows us to shrink $U$ if necessary  so that there exist  $m,M>0$ such that 
\begin{equation}\label{cm_scalar_k_bounds}
        m\le\iota  \omega(h,\delta)\le M \text{ for } (h,\delta) \in U.
\end{equation}
Choose $\eta_0>0$ small enough that $[-\eta_0/2,\eta_0/2] \times [0,\eta_0^2] \subset U$.  For each  $0 < \eta <\eta_0$ we define $T_\eta:(-\eta/2,\eta/2)\to\mathbb{R}$ via
\begin{equation}\label{cm_separation_integral}
        T_\eta(h)=\int_0^h\frac{dr}{\omega(r,\eta^2)(\eta^2/4-r^2)}.
\end{equation}
Then $T_\eta$ is $C^4$ and satisfies $\operatorname{sgn}(T_\eta') = \iota$, which means that $T_\eta$ is strictly monotone. Moreover, \eqref{cm_scalar_k_bounds} and the identity
\begin{equation}\label{cm_log_divergence}
        \int_0^h\frac{dr}{\eta^2/4-r^2}  = \frac{1}{\eta}\log\left(\frac{\eta/2+h}{\eta/2-h}\right)
\end{equation}
show that $T_\eta(h)\to \pm \iota\infty$ as $h \to \pm \eta/2$, so $T_\eta$ is a bijection to $\mathbb{R}$.  This allows us to define $h_\eta = T_\eta^{-1} : \mathbb{R} \to (-\eta/2,\eta/2)$, which is then $C^4$, satisfies  \eqref{cm_scalar_limits} and uniquely solves the IVP \eqref{cm_scalar_eta}.  

\emph{Step 2: Uniform approximation.} Define the rescaled solution $z_\eta: \mathbb{R} \to (-1,1)$ via $z_\eta(s)=2 \eta^{-1} h_\eta(\iota s/\eta)$, which satisfies 
\begin{equation}\label{cm_rescaled_equation}
        z_\eta'=b_\eta(z_\eta)(1-z_\eta^2) \text{ and } z_\eta(0) =0 \text{ for } b_\eta(z)=\frac\iota2 \omega\left(\frac{\eta z}{2},\eta^2\right).
\end{equation}
We claim that, upon shrinking $\eta_0$ if necessary, the map $b_\eta$ satisfies the bound 
\begin{equation}\label{cm_b_estimate}
        \|b_\eta-|\omega_0|/2\|_{C^2([-1,1])}\le C\eta
\end{equation}
for $C>0$ a constant independent of $\eta$. Indeed, the mean value theorem bounds $|b_\eta-|\omega_0|/2|\le C(\eta+\eta^2)$, and decreasing  $\eta_0$ if necessary, we may  assume that $|\omega_0|/4\le b_\eta\le3|\omega_0|/4$ on $[-1,1]$.   Then differentiation shows that $b_\eta'(z)=\iota\eta \omega_h(\eta z/2,\eta^2)/4$ and $b_\eta''(z)=\iota\eta^2 \omega_{hh}(\eta z/2,\eta^2)/8$, and the claim follows readily.  

Next, define $r_\eta= \operatorname{arctanh}(z_\eta)$ and note that  \eqref{cm_rescaled_equation} implies that $r_\eta' =b_\eta(z_\eta)$  and $r_\eta(0)=0$.  Integration and \eqref{cm_b_estimate} then show that
\begin{equation}\label{cm_artanh_bounds}
        |r_\eta(s)-|\omega_0|s/2|\le C\eta|s|, \quad
         |r_\eta(s)|\ge\frac{|\omega_0|}4|s|,
        \text{ and } \operatorname{sgn}(r_\eta(s))=\operatorname{sgn}(s)
         \quad \text{ for } s \in \mathbb{R}.
\end{equation}
The latter two of these imply that every point between $r_\eta(s)$ and $|\omega_0|s/2$ has absolute value at least $|\omega_0||s|/4$. Since $z_\eta=\tanh r_\eta$, the mean value theorem allows us to estimate
\begin{equation}\label{cm_uniform_tanh}
|z_\eta(s)-\tanh(|\omega_0|s/2)| \le C\eta|s|\operatorname{sech}^2(|\omega_0||s|/4) \le C\eta,
\end{equation}
where in the last bound we have also used the boundedness of the map  $s\mapsto |s|\operatorname{sech}^2(|\omega_0||s|/4)$.  Since $h_\eta(t)=\eta z_\eta(\iota\eta t)/2$ and $\iota|\omega_0|/2=\omega_0/2$, the estimate \eqref{cm_scalar_approximation} follows.

\emph{Step 3: Exponential decay.} Combining \eqref{cm_artanh_bounds} with the elementary inequalities $1-\tanh u\le2e^{-2u}$ for $u\ge0$ and $\operatorname{sech}^2u\le4e^{-2|u|}$ for $u \in \mathbb{R}$ shows that
\begin{equation}\label{cm_z_tails}
        |z_\eta(s)-\operatorname{sgn}(s)|\le 2e^{-|\omega_0||s|/2}
        \text{ and }
        1-z_\eta(s)^2\le 4e^{-|\omega_0||s|/2}.
\end{equation}
Set $f_\eta(z)=b_\eta(z)(1-z^2)$. The bound \eqref{cm_b_estimate} shows that $f_\eta,f_\eta',f_\eta''$ are uniformly bounded on $[-1,1]$, while \eqref{cm_z_tails} yields the bound $|f_\eta(z_\eta(s))|\le Ce^{-|\omega_0||s|/2}$. Differentiating \eqref{cm_rescaled_equation} gives $z_\eta'=f_\eta(z_\eta)$, $z_\eta'' =f_\eta'(z_\eta)f_\eta(z_\eta)$, and $z_\eta'''= \left(f_\eta''(z_\eta)f_\eta(z_\eta) +[f_\eta'(z_\eta)]^2\right)f_\eta(z_\eta)$.  Thus, $|z_\eta^{(j)}(s)|\le Ce^{-|\omega_0||s|/2}$ for $j\in\{1,2,3\}$.  Rescaling by $h_\eta(t)=\eta z_\eta(\iota\eta t)/2$ then proves \eqref{cm_scalar_tail_bounds}, with $c=|\omega_0|/2$. 
\end{proof}

We are now ready to prove the main theorem of this section, which constructs solutions to \eqref{cm_height_equation}.  We will then be able to deduce  Theorem \ref{main_small_monotone_bores} as a consequence afterward.

\begin{thm}\label{cm_main_theorem} 
Fix $a,\mu>0$, $\sigma\ge0$, and $g\ge0$ with $g\neq8$, and set $\iota=\operatorname{sgn}(8-g)\in\{-1,1\}$. There exist $A_0\in(0,1)$, constants $C,c>0$, and an open neighborhood $\mathcal{U} \subseteq \mathbb{R}^3$ of $(1/2,0,0)$ such that for every $A\in(A_0,1)$, equation \eqref{cm_height_equation} has a positive solution $H\in C_b^3(\mathbb{R})\cap C^\infty(\mathbb{R})$ satisfying the following. 
\begin{enumerate} 
\item We have $H(0)=1/2$, $\iota H'(t)>0$ for every $t\in\mathbb{R}$, and
\begin{equation}\label{cm_main_limits}
\lim_{t\to\pm\infty}\bigl(H(\iota t),H'(\iota t),H''(\iota t)\bigr)=(H_\pm(A),0,0).
\end{equation}

\item The trajectory $(H,H',H'')$ lies in $\mathcal{U}$.  Any other positive heteroclinic solution of \eqref{cm_height_equation} connecting $H_-$ and $H_+$ with phase trajectory lying entirely in $\mathcal{U}$ is a translate of $H$.  

\item The solution satisfies the uniform approximation
\begin{equation}\label{cm_main_profile}
 \sup_{t\in\mathbb{R}} \left|H(t)-\frac{1}{2}-\frac{(1-A)^{1/2}}{2} \tanh\left(\frac{8a(1-A)^{1/2} t}{8-g}\right)\right|  \le C(1-A).
\end{equation}
Moreover, for $t\in \mathbb{R}$ we have the estimates 
\begin{equation}\label{cm_main_tail}
\left|H(t)-\frac{1+\iota (1-A)^{1/2} \operatorname{sgn}(t)}{2}\right| \le C(1-A)^{1/2} \exp(-c (1-A)^{1/2} |t|),
\end{equation}
and 
\begin{equation}\label{cm_main_derivative_tails}
|H^{(j)}(t)|\le C(1-A)^{(j+1)/2} \exp(-c (1-A)^{1/2} |t|) \text{ for } j \in \{1,2,3\}.
\end{equation}
\end{enumerate} 
\end{thm}
\begin{proof}

We first consider the case $\sigma =0$.  Writing $H=1/2+h$ and $A=1-\delta$, the ODE \eqref{cm_height_equation} is equivalent near $(h,\delta)=(0,0)$ to the two-dimensional system
\begin{equation}\label{cm_zero_capillarity_system}
    h'=v \text{ and }
    v'=\frac{v^2}{H} - \frac{A^2/H-gH^2}{4\mu A}v +\frac{a}{\mu A}\left(\frac\delta4-h^2\right).
\end{equation}
The linearization of this problem at $(h,v,\delta)=(0,0,0)$ has eigenvalues $0$ and $-(8-g)/(16\mu)$. Augmenting by $\delta'=0$ and applying the center manifold theorem as in the proof of Lemma \ref{cm_reduction}, we may produce  a $C^4$ graph $v=P(h,\delta)$ such that the following three conditions hold.  First, we have the identities
\begin{equation}
    P(0,0)=P_h(0,0)=0  \text{ and } P_\delta(0,0)=\frac{4a}{8-g}.
\end{equation}
Second, there exist an open neighborhood $\mathcal{O}$ of $(0,0)$ and $\delta_0>0$ such that every complete trajectory of \eqref{cm_zero_capillarity_system} contained in $\mathcal{O}$, with fixed $|\delta|<\delta_0$, lies on the graph.  Third, the nearby equilibria again satisfy $v=0$ and $\delta=4h^2$, so the division argument used in Lemma \ref{cm_reduction} yields a $C^3$ map $k$, defined in a neighborhood of the origin, such that 
\begin{equation}\label{cm_zero_capillarity_reduction}
    P(h,\delta)=k(h,\delta)\left(\frac\delta4-h^2\right) \text{ and } k(0,0)=\frac{16a}{8-g}.
\end{equation}
Lemma \ref{cm_scalar_lemma} now constructs a profile $h$ for all $A<1$ sufficiently close to unity, such that $H = h+ 1/2$ satisfies  the normalization, monotonicity, limits, and estimates of the first and third items.  Smoothness follows from \eqref{cm_zero_capillarity_system} and a standard bootstrap. 

Now choose a bounded open neighborhood $\mathcal U$ of $(1/2,0,0)$ such that $1/4<H<3/4$ on $\overline{\mathcal U}$ and $(H,v,w)\in\mathcal U$ implies $(H-1/2,v)\in\mathcal O$. Since
\begin{equation}
    H'=P(h,1-A)  \text{ and } H''=P_h(h,1-A)P(h,1-A),
\end{equation}
the constructed trajectory lies in $\mathcal U$ for $A$ sufficiently close to unity. Any other heteroclinic with phase trajectory in $\mathcal U$ lies on the same graph and crosses $H=1/2$ with $H'=P(0,1-A)$. Uniqueness for \eqref{cm_zero_capillarity_system} thus proves uniqueness up to translation.  Choosing $A_0\in(0,1)$ sufficiently close to unity to satisfy all these restrictions completes the proof in the case $\sigma=0$.

We now assume that $\sigma>0$ and divide the remainder of the proof into three steps.

\emph{Step 1: Choice of the neighborhoods.} Let $\mathcal O,D_0,\delta_0,P,Q,k$ be as in Lemma \ref{cm_reduction}. Choose a bounded open neighborhood $\mathcal U$ of $(1/2,0,0)$ and a number $0<\delta_1<\delta_0$ small enough that $1/4<H<3/4$ on $\overline{\mathcal U}$ and
\begin{equation}\label{cm_neighborhood_choice}
    (H,v,w)\in\mathcal U \text{ and } |\delta|<\delta_1
     \Rightarrow (H-1/2,v,w)\in\mathcal O \text{ and } (H-1/2,\delta)\in D_0.
\end{equation}

\emph{Step 2: Construction and estimates.} Set $k_0=k(0,0)=16a/(8-g)\neq0$. Lemma \ref{cm_scalar_lemma}, applied with $\omega=k$, provides a scalar solution $h_\eta$ for every sufficiently small $\eta>0$. Set
\begin{equation}
    H(t)=\frac{1}{2}+h_\eta(t)\text{ and }A=1-\eta^2.
\end{equation}
Restrict $\eta$ so that $\eta^2<\delta_1$ and $(h,\eta^2)\in D_0$ whenever $|h|\le\eta/2$. Since $|h_\eta(t)|<\eta/2$ for every $t\in\mathbb{R}$, Lemma \ref{cm_reduction} lifts this scalar solution to a solution of \eqref{cm_extended_system} with parameter $\delta=\eta^2$, with
\begin{equation}
    H'=P(h_\eta,\eta^2)\text{ and }H''=Q(h_\eta,\eta^2).
\end{equation}
The continuity of $P,Q$ and the identities $P(0,0)=Q(0,0)=0$ show that
\begin{equation}\label{cm_graph_uniform_smallness}
    \sup_{|h|\le\eta/2}\left(|h|+|P(h,\eta^2)|+|Q(h,\eta^2)|\right) \to 0 \text{ as }\eta\to0.
\end{equation}
Further restricting $\eta$ if necessary, we conclude that the entire trajectory $(H,H',H'')$ lies in $\mathcal U$. In particular, $H$ takes values in $(0,1)$, so Lemma \ref{cm_linearization} shows that $H$ solves \eqref{cm_height_equation}. Smoothness follows by bootstrapping \eqref{cm_extended_system}.

Since $\operatorname{sgn}(k_0)=\iota$, the normalization, strict monotonicity, and limits in the first item follow from Lemma \ref{cm_scalar_lemma}, with $H_\pm(A)=(1\pm\eta)/2$. The estimates in the third item follow from the same lemma upon substituting $\eta=(1-A)^{1/2}$ and $k_0/2=8a/(8-g)$. The constants $C,c>0$ in these estimates are independent of sufficiently small $\eta$. In particular, $H\in C_b^3(\mathbb{R})$.

\emph{Step 3: Local uniqueness.} Suppose that $\widetilde H$ is another positive heteroclinic solution with the same $A$, connecting $H_-(A)$ and $H_+(A)$, and with $(\widetilde H,\widetilde H',\widetilde H'')$ contained in $\mathcal U$. By Lemma \ref{cm_linearization} and \eqref{cm_neighborhood_choice}, the trajectory $\widetilde x=(\widetilde H-1/2,\widetilde H',\widetilde H'')$ is a solution of \eqref{cm_extended_system} defined on $\mathbb{R}$ and taking values in $\mathcal O$. Lemma \ref{cm_reduction}, with the fixed parameter $\delta=\eta^2<\delta_1$, therefore shows that $\widetilde x(\mathbb{R})\subseteq\mathcal M_{\eta^2}$. Since $\widetilde H$ connects $H_-(A)<1/2$ and $H_+(A)>1/2$, there exists $t_0\in\mathbb{R}$ such that $\widetilde H(t_0)=1/2$. The graph representation then gives
\begin{equation}\label{cm_uniqueness_phase}
        (\widetilde H(t_0),\widetilde H'(t_0),\widetilde H''(t_0))    =(1/2,P(0,\eta^2),Q(0,\eta^2))
        =(H(0),H'(0),H''(0)).
\end{equation}
Uniqueness for the smooth first order system implies that $\widetilde H(t_0+t)=H(t)$ for every $t\in\mathbb{R}$, proving the second item.

Finally, choose a common $\eta_0\in(0,1)$ satisfying all the restrictions above, including $\eta_0^2<\delta_1$, and set $A_0=1-\eta_0^2$. Every $A\in(A_0,1)$ corresponds to $\eta=(1-A)^{1/2}\in(0,\eta_0)$, so the construction and estimates apply with constants independent of $A$. This proves the theorem.
\end{proof}

Note that the leading order profile identified in the theorem depends on $A$, $a$, and $g$, but is independent of $\mu$ and $\sigma$.  With a bit more work we can refine the profile to see the first contributions of viscosity and capillarity.  We record this now.

\begin{prop}\label{cm_higher_order_profiles}
Fix $a,\mu>0$, $\sigma\ge0$, and $g\ge0$ with $g\ne8$. There exist $A_0\in(0,1)$ and $C>0$ such that for every $A\in(A_0,1)$ the solutions $H_\sigma$ provided by Theorem \ref{cm_main_theorem} with $\sigma \ge 0$ satisfy the following.
\begin{enumerate}
\item We have the refined approximation
\begin{multline}\label{cm_second_order_profile}
    \sup_{t\in\mathbb R}\Bigg|H_\sigma(t)-\frac12
    -\frac{(1-A)^{1/2}}2
       \tanh\left(\frac{8a(1-A)^{1/2}t}{8-g}\right)
    \\
    -(1-A)\left(\frac{g+4}{8-g}
                    +\frac{128a\mu}{(8-g)^2}\right)
       \operatorname{sech}^2\left(\frac{8a(1-A)^{1/2}t}{8-g}\right)
       \log\cosh\left(\frac{8a(1-A)^{1/2}t}{8-g}\right)
       \Bigg|\le C(1-A)^{3/2}.
\end{multline}
\item The difference between the capillary and noncapillary  profiles satisfies the bound
\begin{multline}\label{cm_capillary_height_correction}
    \sup_{t\in\mathbb R}\Bigg|H_\sigma(t)-H_0(t)
    -\frac{64\sigma a^2(1-A)^{3/2}}{(8-g)^3}
       \operatorname{sech}^2\left(\frac{8a(1-A)^{1/2}t}{8-g}\right)
    \\
    {}\times\left[
       3\tanh\left(\frac{8a(1-A)^{1/2}t}{8-g}\right)
       -\frac{16a(1-A)^{1/2}t}{8-g}\right]
       \Bigg|\le C(1-A)^2.
\end{multline}
Consequently, writing $\rho_\sigma=\log H_\sigma$ and $\rho_0=\log H_0$, we have
\begin{multline}\label{cm_capillary_log_correction}
    \sup_{t\in\mathbb R}\Bigg|\rho_\sigma(t)-\rho_0(t)
    -\frac{128\sigma a^2(1-A)^{3/2}}{(8-g)^3}
       \operatorname{sech}^2\left(\frac{8a(1-A)^{1/2}t}{8-g}\right)
    \\
    {}\times\left[
       3\tanh\left(\frac{8a(1-A)^{1/2}t}{8-g}\right)
       -\frac{16a(1-A)^{1/2}t}{8-g}\right]
       \Bigg|\le C(1-A)^2.
\end{multline}
\end{enumerate}
\end{prop}
\begin{proof}
Write $\delta=1-A$ so that for $\sigma \ge 0$
\begin{equation}\label{P_factorization}
    P_\sigma(h,\delta)  =k_\sigma(h,\delta)\left(\frac\delta4-h^2\right)
    \text{ and } k_\sigma(0,0)=\frac{16a}{8-g},
\end{equation}
where $P_\sigma$ and $k_\sigma$ denote the  functions  provided by Lemma \ref{cm_reduction} when $\sigma>0$ and by \eqref{cm_zero_capillarity_reduction} when $\sigma=0$. 

Since $\delta$ is fixed along each trajectory, the reduced equation $h'=P_\sigma(h,\delta)$ allows us to set $H = h+1/2$ and compute
\begin{equation}
    H''=(\partial_h P_\sigma)P_\sigma    \text{ and }
    H'''=(\partial_h^2 P_\sigma)P_\sigma^2    +(\partial_h P_\sigma)^2P_\sigma.
\end{equation}
Plugging this into \eqref{cm_height_third_derivative} yields
\begin{equation}
    \sigma\left((\partial_h^2 P_\sigma)P_\sigma^2 +(\partial_h P_\sigma)^2P_\sigma\right)
    =-\frac{4\mu A}{H^2}(\partial_h P_\sigma)P_\sigma
      +\frac{4\mu A}{H^3}P_\sigma^2 -\left(\frac{A^2}{H^3}-g\right)P_\sigma
      +\frac{4a}{H^2}\left(\frac{\delta}{4}-h^2\right),
\end{equation}
where $P_\sigma$ and its derivatives are understood to be evaluated at $(h,\delta)$. This identity remains valid when $\sigma=0$.  Next, we plug \eqref{P_factorization} into the previous identity, divide  by $\delta/4-h^2\neq 0$, and collect terms; this reveals that
\begin{multline}\label{cm_reduced_invariance}
    \sigma\Bigg[ \left(k_\sigma^2\partial_h^2k_\sigma +k_\sigma(\partial_h k_\sigma)^2\right)
    \left(\frac\delta4-h^2\right)^2 -8hk_\sigma^2(\partial_h k_\sigma) \left(\frac\delta4-h^2\right)   +k_\sigma^3\left(6h^2-\frac\delta2\right)\Bigg]
    \\
    =-\frac{4\mu A}{H^2}k_\sigma(\partial_h k_\sigma) \left(\frac\delta4-h^2\right) +\frac{8\mu A}{H^2}hk_\sigma^2       +\frac{4\mu A}{H^3}k_\sigma^2 \left(\frac\delta4-h^2\right) -\left(\frac{A^2}{H^3}-g\right)k_\sigma +\frac{4a}{H^2}.
\end{multline}
In deriving this we assumed $\delta/4-h^2 \neq 0$, but it remains valid at equality by continuity. We now use this to compute $\partial_h k_\sigma(0,0)$; indeed, differentiating 
\eqref{cm_reduced_invariance} with respect to $h$  and evaluating at $(h,\delta)=(0,0)$ provides the identity
\begin{equation}
0=32\mu k_\sigma(0,0)^2+48k_\sigma(0,0)      -(8-g)\partial_h k_\sigma(0,0)-64a.
\end{equation}
Combining this with $k_\sigma(0,0)=16a/(8-g)$, we obtain
\begin{equation}
    (8-g)\partial_h k_\sigma(0,0)
    =4(g+4)k_\sigma(0,0)+32\mu k_\sigma(0,0)^2,
\end{equation}
and hence
\begin{equation}
    \partial_h k_\sigma(0,0)
    =\frac{64a(g+4)}{(8-g)^2}
       +\frac{8192\mu a^2}{(8-g)^3}.
\end{equation}
Thus both $k_\sigma(0,0)$ and $\partial_h k_\sigma(0,0)$ are independent of $\sigma$, and Taylor's theorem allows us to estimate 
\begin{equation}\label{cm_reduced_first_correction}
k_\sigma(h,\delta)=\frac{16a}{8-g} +\left(\frac{64a(g+4)}{(8-g)^2} +\frac{8192\mu a^2}{(8-g)^3}\right)h  +O(\delta)
\end{equation}
uniformly for $|h|\le\sqrt{\delta}/2$ since the Taylor remainder is $O(h^2+|\delta|)$.

We next subtract the identity \eqref{cm_reduced_invariance} evaluated at $\sigma=0$ from it evaluated at $\sigma >0$, respecting the left-right ordering.   Since $k_\sigma$ and $k_0$ have the same value and first $h$-derivative at the origin, Taylor's theorem shows that
\begin{equation}
    k_\sigma-k_0=O(\delta)
    \text{ and }
    \partial_h k_\sigma-\partial_h k_0=O(\sqrt\delta)
\end{equation}
uniformly for $|h|\le\sqrt\delta/2$, which then implies that 
\begin{equation}
\begin{aligned}
    \left(\frac\delta4-h^2\right)
       \left(k_\sigma\partial_h k_\sigma -k_0\partial_h k_0\right)
       &=O(\delta^{3/2}),\\
    h(k_\sigma^2-k_0^2)&=O(\delta^{3/2}),\\
    \left(\frac\delta4-h^2\right)(k_\sigma^2-k_0^2)
       &=O(\delta^2), \\
        \left(\frac{A^2}{H^3}-g\right)(k_\sigma-k_0) &=(8-g)(k_\sigma-k_0)+O(\delta^{3/2}).
\end{aligned}
\end{equation}
These estimates account for the differences of the first four terms on the right side of \eqref{cm_reduced_invariance}, and the fifth term involving $4a/H^2$ cancels in the difference.
On the left of \eqref{cm_reduced_invariance}, the first two terms are $O(\delta^2)$ and $O(\delta^{3/2})$, respectively. Since $k_\sigma=16a/(8-g)+O(\sqrt\delta)$, subtraction therefore leaves
\begin{equation}
    \sigma\left(\frac{16a}{8-g}\right)^3
       \left(6h^2-\frac\delta2\right)
       =-(8-g)(k_\sigma-k_0)+O(\delta^{3/2}).
\end{equation}
Solving for the difference then shows that
\begin{equation}\label{cm_reduced_capillary_correction}
    k_\sigma(h,\delta)-k_0(h,\delta) =\frac{2048\sigma a^3}{(8-g)^4}(\delta-12h^2)
       +O(\delta^{3/2})
\end{equation}
uniformly for $|h|\le\sqrt\delta/2$.

To transfer these expansions of $k_\sigma$ to the profiles, we next define
\begin{equation}
    r_\sigma(t)=\operatorname{arctanh} \left(\frac{2H_\sigma(t)-1}{\sqrt\delta}\right) \text{ for } t \in \mathbb{R}.
\end{equation}
The reduced equation implies that $r_\sigma$ solves the IVP
\begin{equation}
    r_\sigma'(t)=\frac{\sqrt\delta}{2}  k_\sigma\left(\frac{\sqrt\delta}{2}\tanh r_\sigma(t),\delta\right)
    \text{ and } r_\sigma(0)=0,
\end{equation}
which may be solved implicitly via the integrated identity
\begin{equation}\label{cm_refined_phase_identity}
    \frac{8a\sqrt\delta\,t}{8-g}
    =\int_0^{r_\sigma(t)}
       \frac{16a}
       {(8-g)k_\sigma(\sqrt\delta\tanh s/2,\delta)}\,ds.
\end{equation}
The integrand is $1+O(\sqrt\delta)$ uniformly in $s$, and hence lies between $1/2$ and $3/2$ for sufficiently small $\delta$. It follows that $r_\sigma(t)$ has the same sign as $t/(8-g)$ and that
\begin{equation}\label{cm_refined_phase_bounds}
    c\sqrt\delta|t|\le |r_\sigma(t)|\le C\sqrt\delta|t|
    \text{ and }
    \left|r_\sigma(t)-\frac{8a\sqrt\delta\,t}{8-g}\right|
       \le C\sqrt\delta\,|r_\sigma(t)|
       \le C\delta|t|.
\end{equation}

Using \eqref{cm_reduced_first_correction} to expand the reciprocal in \eqref{cm_refined_phase_identity} implies that
\begin{equation}
    \frac{16a}{(8-g)k_\sigma(\sqrt\delta\tanh s/2,\delta)}
    =1-\sqrt\delta\left(\frac{2(g+4)}{8-g}  +\frac{256a\mu}{(8-g)^2}\right)\tanh s +O(\delta),
\end{equation}
and integrating this expression shows that
\begin{equation}
    \frac{8a\sqrt\delta\,t}{8-g}     =r_\sigma(t) -\sqrt\delta\left(\frac{2(g+4)}{8-g}+\frac{256a\mu}{(8-g)^2}\right) \log\cosh r_\sigma(t) +O(\delta|r_\sigma(t)|).
\end{equation}
Since the derivative of $\log\cosh$ is bounded by one,
\begin{equation}
    \left|\log\cosh r_\sigma(t) -\log\cosh\left(\frac{8a\sqrt\delta\,t}{8-g}\right)\right|
       \le \left|r_\sigma(t)-\frac{8a\sqrt\delta\,t}{8-g}\right| \le C\delta|t|,
\end{equation}
and upon plugging this into the previous identity we deduce that
\begin{equation}\label{cm_refined_phase_expansion}
    r_\sigma(t)=\frac{8a\sqrt\delta\,t}{8-g} +\sqrt\delta\left(\frac{2(g+4)}{8-g} +\frac{256a\mu}{(8-g)^2}\right) \log\cosh\left(\frac{8a\sqrt\delta\,t}{8-g}\right)
    +O(\delta^{3/2}|t|).
\end{equation}

Next, we use the identity $k_\sigma^{-1}-k_0^{-1}=-(k_\sigma-k_0)/(k_\sigma k_0)$ together with 
 \eqref{cm_reduced_capillary_correction} to obtain the equality 
\begin{equation}
\frac{16a}{(8-g)k_\sigma(\sqrt\delta\tanh s/2,\delta)} -\frac{16a}{(8-g)k_0(\sqrt\delta\tanh s/2,\delta)} = -\frac{128\sigma a^2\delta}{(8-g)^3} (1-3\tanh^2s)+O(\delta^{3/2}),
\end{equation}
uniformly in $s$. Subtracting \eqref{cm_refined_phase_identity} with $\sigma =0$ from it with $\sigma >0$ and using the computation
\begin{equation}
    \int_0^r(1-3\tanh^2s)\,ds=3\tanh r-2r
\end{equation}
then shows that
\begin{equation}
    \int_{r_0(t)}^{r_\sigma(t)}\frac{16a} {(8-g)k_\sigma(\sqrt\delta\tanh s/2,\delta)}\,ds
    =\frac{128\sigma a^2\delta}{(8-g)^3} \bigl(3\tanh r_0(t)-2r_0(t)\bigr)
    +O (\delta^{3/2}|r_0(t)|).
\end{equation}
The integrand on the left is at least $1/2$, while $|3\tanh r-2r|\le5|r|$, and  hence this identity first shows that
\begin{equation}\label{cm_phase_difference_bound}
    |r_\sigma(t)-r_0(t)|  \le C\delta|r_0(t)| \le C\delta^{3/2}|t|.
\end{equation}
On the other hand, we can also replace  the integrand on the left by unity, which introduces  an error bounded by
\begin{equation}
    C\sqrt\delta\,|r_\sigma(t)-r_0(t)| \le C\delta^2|t|.
\end{equation}
Also, the derivative of $3\tanh r-2r$ is bounded by two, so replacing $r_0(t)$ on the right by $8a\sqrt\delta\,t/(8-g)$ introduces another error of this order, by \eqref{cm_refined_phase_bounds}. Combining these observations, we deduce  that
\begin{equation}\label{cm_refined_phase_difference}
r_\sigma(t)-r_0(t) = \frac{128\sigma a^2\delta}{(8-g)^3} \left[3\tanh\left(\frac{8a\sqrt\delta\,t}{8-g}\right) -\frac{16a\sqrt\delta\,t}{8-g}\right]    +O(\delta^2|t|).
\end{equation}

It remains to convert the above estimates into uniform estimates for the profiles.  For sufficiently small $\delta$ it holds that  $r_\sigma(t)$, $r_0(t)$, and $8a\sqrt\delta\,t/(8-g)$ all have the same sign, and every point in the interval containing them has absolute value at least $c\sqrt\delta|t|$.  The first two derivatives of $\tanh$ on this interval are thus bounded by $Ce^{-c\sqrt\delta|t|}$. Taylor's theorem and \eqref{cm_refined_phase_bounds} then allow us to bound 
\begin{equation}
\tanh r_\sigma(t) =\tanh\left(\frac{8a\sqrt\delta\,t}{8-g}\right) + \operatorname{sech}^2\left(\frac{8a\sqrt\delta\,t}{8-g}\right)  \left(r_\sigma(t)-\frac{8a\sqrt\delta\,t}{8-g}\right)  +O\left(\delta^2t^2e^{-c\sqrt\delta|t|}\right).
\end{equation}
Substituting \eqref{cm_refined_phase_expansion} here and multiplying by $\sqrt\delta/2$ proves  \eqref{cm_second_order_profile}  since
\begin{equation}
    H_\sigma(t)=\frac12+\frac{\sqrt\delta}{2}\tanh r_\sigma(t).
\end{equation}
More precisely, the resulting error is bounded by
\begin{equation}
    C\left(\delta^2|t|+\delta^{5/2}t^2\right)e^{-c\sqrt\delta|t|}
       \le C\delta^{3/2}(1+\delta t^2)e^{-c\sqrt\delta|t|}.
\end{equation}

For the difference of the profiles, the fundamental theorem of calculus allows us to compute
\begin{multline}
    \tanh r_\sigma(t)-\tanh r_0(t)
       -\operatorname{sech}^2\left(\frac{8a\sqrt\delta\,t}{8-g}\right)
          \bigl(r_\sigma(t)-r_0(t)\bigr)
    =\int_{r_0(t)}^{r_\sigma(t)}
       \left[ \operatorname{sech}^2u
          -\operatorname{sech}^2\left(\frac{8a\sqrt\delta\,t}{8-g}\right)
       \right]\,du.
\end{multline}
For $u$ lying between $r_0(t)$ and $r_\sigma(t)$, every point between $u$ and $8a\sqrt\delta\,t/(8-g)$ has absolute value at least $c\sqrt\delta|t|$. Since $(\operatorname{sech}^2)'=-2\operatorname{sech}^2\tanh$, the mean value theorem and \eqref{cm_refined_phase_bounds} give
\begin{equation}
\left|\operatorname{sech}^2u -\operatorname{sech}^2\left(\frac{8a\sqrt\delta\,t}{8-g}\right)\right|
    \le Ce^{-c\sqrt\delta|t|}   \left|u-\frac{8a\sqrt\delta\,t}{8-g}\right|
    \le C\delta|t|e^{-c\sqrt\delta|t|}.
\end{equation}
Consequently, \eqref{cm_phase_difference_bound} implies that
\begin{equation}
    \left|\int_{r_0(t)}^{r_\sigma(t)}
       \left[\operatorname{sech}^2u
       -\operatorname{sech}^2\left(\frac{8a\sqrt\delta\,t}{8-g}\right)\right]
       \,du\right|
    \le C\delta|t|e^{-c\sqrt\delta|t|} |r_\sigma(t)-r_0(t)|
    \le C\delta^{5/2}t^2e^{-c\sqrt\delta|t|}.
\end{equation}
Since
\begin{equation}
    H_\sigma(t)-H_0(t) =\frac{\sqrt\delta}{2} \bigl(\tanh r_\sigma(t)-\tanh r_0(t)\bigr),
\end{equation}
the preceding estimates give
\begin{equation}
    H_\sigma(t)-H_0(t) =\frac{\sqrt\delta}{2} \operatorname{sech}^2\left(\frac{8a\sqrt\delta\,t}{8-g}\right) \bigl(r_\sigma(t)-r_0(t)\bigr)
    +O\left(\delta^3t^2e^{-c\sqrt\delta|t|}\right).
\end{equation}
We now substitute \eqref{cm_refined_phase_difference} into this
formula. The $O(\delta^2|t|)$ remainder in the phase difference
contributes at most
$C\delta^{5/2}|t|e^{-c\sqrt\delta|t|}$, so
\begin{multline}
    H_\sigma(t)-H_0(t)
    =\frac{64\sigma a^2\delta^{3/2}}{(8-g)^3}
       \operatorname{sech}^2\left(\frac{8a\sqrt\delta\,t}{8-g}\right)
       \left[ 3\tanh\left(\frac{8a\sqrt\delta\,t}{8-g}\right)
       -\frac{16a\sqrt\delta\,t}{8-g}\right]
    \\
    +O\left(
       \left(\delta^{5/2}|t|+\delta^3t^2\right)
       e^{-c\sqrt\delta|t|}\right).
\end{multline}
The remainder is bounded by $C\delta^2(1+\delta t^2)e^{-c\sqrt\delta|t|}$, which is at most
$C\delta^2$ uniformly on $\mathbb R$. Recalling that $\delta=1-A$, we obtain \eqref{cm_capillary_height_correction}.

Finally, recall that $\rho_\sigma=\log H_\sigma$ and $\rho_0=\log H_0$. Since $H_\sigma,H_0=1/2+O(\sqrt\delta)$ uniformly, the derivative of $\log$ at every point between
$H_\sigma(t)$ and $H_0(t)$ is $2+O(\sqrt\delta)$. The mean value theorem and
\eqref{cm_capillary_height_correction} therefore show that
\begin{equation}
    \left|\rho_\sigma(t)-\rho_0(t) -2\bigl(H_\sigma(t)-H_0(t)\bigr)\right|
    \le C\sqrt\delta\,|H_\sigma(t)-H_0(t)|
    \le C\delta^2
\end{equation}
uniformly on $\mathbb R$. Combining this estimate with \eqref{cm_capillary_height_correction} proves \eqref{cm_capillary_log_correction}.
\end{proof}

We now have all the tools needed to present the 

\begin{proof}[Proof of Theorem \ref{main_small_monotone_bores}]
Let $H$ and $\mathcal U$ be as in Theorem \ref{cm_main_theorem}, and set $\mathcal V=\mathcal U$. Equations \eqref{ss_height_equation} and \eqref{cm_height_equation} are equivalent, so the first three items follow directly from that theorem, after renaming the constants. As in Step 5 of the proof of Theorem \ref{main_small_sigma_persistence_in_C}, repeated differentiation of the smooth first order formulation of the ODE shows that $H\in C_b^\infty(\mathbb{R})$; here we use \eqref{cm_zero_capillarity_system} when $\sigma=0$ and \eqref{cm_extended_system} when $\sigma>0$.

Increasing $A_0$ if necessary, we may apply Proposition \ref{cm_higher_order_profiles} with $H_\sigma=H$ and the $\sigma=0$ profile denoted $H_0$. Estimates \eqref{cm_second_order_profile} and \eqref{cm_capillary_height_correction} then prove the fourth and fifth items, enlarging $C_0$ if necessary.
\end{proof}

%%%%%%%%%%%%%%%%%%%%%%%%%%%%%%%%%%%%%%%%%%%%%%%%%
%%%%%%%%%%%%%%%%%%%%%%%%%%%%%%%%%%%%%%%%%%%%%%%%%
%%%%%%%%%%%%%%%%%%%%%%%%%%%%%%%%%%%%%%%%%%%%%%%%%
\section*{Acknowledgments}
%%%%%%%%%%%%%%%%%%%%%%%%%%%%%%%%%%%%%%%%%%%%%%%%%
%%%%%%%%%%%%%%%%%%%%%%%%%%%%%%%%%%%%%%%%%%%%%%%%%
%%%%%%%%%%%%%%%%%%%%%%%%%%%%%%%%%%%%%%%%%%%%%%%%%
 
The authors acknowledge the use of AI tools in the writing of this paper: they were used for generating the graphics for the figures, checking the calculations, and proofreading the prose.

%%%%%%%%%%%%%%%%%%%%%%%%%%%%%%%%%%%%%%%%%%%%%%%%%%%%%%%%%%%%%%%%%%%%%%%%
%References
%%%%%%%%%%%%%%%%%%%%%%%%%%%%%%%%%%%%%%%%%%%%%%%%%%%%%%%%%%%%%%%%%%%%%%%%

\end{document}